\documentclass[final]{siamltex}
\usepackage[compress]{cite}
\usepackage{amssymb} 
\usepackage[colorlinks,linkcolor=blue,citecolor=green,urlcolor=red]{hyperref}
 \usepackage{graphicx,float}
\usepackage{amsmath,bm}
  \numberwithin{figure}{section}
  \numberwithin{table}{section}
  \usepackage{amsfonts,mathrsfs}
\usepackage{amsmath,amssymb}
\numberwithin{hypothesis}{section}
\newtheorem{remark}{Remark}
\numberwithin{remark}{section}
\numberwithin{theorem}{section}
\newcommand{\Ii}{{{\rm i}}}
\usepackage{multirow}

\newcommand{\diss}{discretization } 
\newcommand{\pd}{preconditioner}
\newcommand{\pdd}{preconditioner }
\newcommand{\dg}{diagonalization}
\newcommand{\dgg}{diagonalization }
\newcommand{\T}{\mathsf{T}}

\newcommand{\CA}{{\mathcal{A}}}

\newcommand{\CF}{{\mathcal{F}}}
\newcommand{\CD}{{\mathbb{D}}}
\newcommand{\CCD}{{\mathcal{D}}}
\newcommand{\BK}{{\mathbb{K}}}

\newcommand{\CP}{{\mathcal{P}}}

\newcommand{\CV}{{\mathcal{V}}}
\newcommand{\IR}{{\mathbb{R}}}
\newcommand{\IC}{{\mathbb{C}}}

\DeclareMathOperator{\Sinc}{Sinc}

\newcommand{\tol}{{\textit{\texttt{tol}}}}
\newcommand{\tn}[1]{\textnormal{#1}\ } 
\newcommand{\bmt}{\left[ \begin{array}{ccccccccccccccccccccccccccccccccccccc}}
	\newcommand{\emt}{\end{array}\right]}
\usepackage{booktabs}

\usepackage[margin=1.4in]{geometry}
\newcommand{\bmtx}{\left[ \begin{array}{c|cccccccccccccccccccccccccccccccccccc}}
	\newcommand{\emtx}{\end{array}\right]}

\newcommand{\bean}{\begin{eqnarray*}}
	\newcommand{\eean}{\end{eqnarray*}}
\newcommand{\bea}{\begin{eqnarray}}
	\newcommand{\eea}{\end{eqnarray}}
\newcommand{\eq}{\begin{equation}\begin{array}{lllllllll}}
		\newcommand{\ee}{\end{array}\end{equation}}
\newcommand{\eqn}{\begin{equation*}\begin{array}{lllllllll}}
		\newcommand{\een}{\end{array}\end{equation*}}

\makeatletter
\renewcommand\normalsize{%
   \@setfontsize\normalsize\@xpt\@xiipt
   \abovedisplayskip 8.5\p@ \@plus2\p@ \@minus5\p@
   \abovedisplayshortskip \z@ \@plus3\p@
   \belowdisplayshortskip 8.5\p@ \@plus3\p@ \@minus3\p@
   \belowdisplayskip \abovedisplayskip
   \let\@listi\@listI}
\makeatother
 
\title{Parallel-in-time  preconditioners  for the Sinc-Nystr\"{o}m method}
 \author{Jun Liu\thanks{Department of Mathematics and Statistics, Southern Illinois University Edwardsville, Edwardsville, IL 62026, USA. E-mail: \texttt{juliu@siue.edu}}\and Shu-Lin Wu\thanks{School of Mathematics and Statistics, Northeast Normal University, Changchun 130024, China.
E-mail: \texttt{wushulin84@hotmail.com}} 
}
 
\begin{document}

\maketitle

\begin{abstract}
The Sinc-Nystr\"{o}m method  is a high-order {numerical  method} based on Sinc basis functions for  discretizing evolutionary differential equations in time. But in this method we have to solve all the time steps in one-shot (i.e. all-at-once), which results in  a large-scale  nonsymmetric dense system that is expensive to handle. 
			In this paper, we propose and analyze preconditioner   for  such dense system arising from  both the parabolic and hyperbolic PDEs.
			The proposed   preconditioner  is  a low-rank  perturbation of the original   matrix   {and  has two advantages. First, we show that  the
			 eigenvalues  of the preconditioned system are   highly clustered with some uniform bounds which are independent of the mesh parameters.   Second,  the preconditioner can be   used parallel for all the Sinc time points via  a block diagonalization procedure.  Such a parallel potential owes to the fact that  the eigenvector matrix of the  diagonalization is well conditioned. In particular, we show that the condition number of  the eigenvector matrix  only mildly grows as the number of Sinc time points increases, and thus the roundoff error    arising from the diagonalization procedure is controllable.}
			 The effectiveness of our proposed PinT preconditioners is verified by the observed  mesh-independent convergence rates of the preconditioned GMRES in reported {numerical} examples.		
 \end{abstract}



\pagestyle{myheadings}
\thispagestyle{plain}

\markboth{J. Liu and S.-L. Wu}{PinT Preconditioners for Sinc-Nystr\"{o}m  method} 
	\section{Introduction}
	   Belonging to the large family of pseudospectral methods,
 	the {Sinc-Nystr\"{o}m} numerical method  \cite{stenger2012numerical} is a    special one  among numerous high-order discretization schemes that
 	can achieve an exponential order of accuracy for approximating ODEs/PDEs and integral equations \cite{Bialecki1988,Bialecki1989,Lin2013,Rahmoune2020}, even in the presence of boundary singularities and boundary
 	layers\footnote[1]{The presence of boundary singularities and layers {often dramatically}  
 	deteriorates the expected approximation accuracy of the standard finite difference and finite element discretization schemes, although such {a} degradation can be mildly alleviated with adaptive meshing or local refined meshing techniques.}. {Such a   method  lies in  first transforming  the initial-value  ODE    into a Volterra integral equation of the second kind and then applying
 	the  collocation   approximation to the latter. } 
 	Besides {the} exponential order
 	of accuracy, the   basis functions provide the computationally favorable Toeplitz structures of
 	the {discretization} matrix $\CA$, which will be used in this paper to facilitate  the development of {efficient \pdd  denoted by $\CP$}.  
{However, for large-scale ODEs (such as the ones  arise from semi-discretizing time-dependent PDEs in high dimension)  the unknowns over all the collocation   time points} are fully coupled and this requires to solve a large-scale nonsymmetric   dense system $\CA{\bm y}_h={\bm b}_h$,  which is often very  time consuming to solve.
 	In this paper, we propose and analyze  a structured   preconditioner for handling this problem. 
	
The novelty of the proposed  \pdd is twofold.  First,  as we  will show in Section \ref{sec4} the eigenvalues of the preconditioned matrix $\CP^{-1}\CA$ are highly clustered for both the parabolic and hyperbolic problems, which  indicates    fast convergence of the  preconditioned GMRES in practice  (confirmed by numerical results in Section \ref{sec5}).
Second,  the \pdd  can be used in parallel for all the Sinc time points. We briefly explain such a parallel-in-time (PinT)  implementation   as follows.   By    diagonalizing  $\CP$ as  $\CP=\CV\CCD\CV^{-1}$ with a block diagonal matrix $\CCD$ and $\CV=V\otimes I_n$,   we can compute $\CP^{-1}{\bm r}$  with any vector ${\bm r}$  via   three steps: 
	$$
	{\bm s}_1:=\CV^{-1}{\bm r}=(V^{-1}\otimes I_n){\bm r},~~~{\bm s}_2:=\CCD^{-1}{\bm s}_1,~~~\CP^{-1}{\bm r}=\CV{\bm s}_2=(V\otimes I_n){\bm s}_2,
	$$ 
	where  $V\in\mathbb{C}^{m\times m}$, $I_n$ and $I_m$ are identity matrices with  $m$ and $n$  being  respectively the number of   collocation time  points and the  dimension of  ODEs.   
 (More details on these three steps will be supplied in Section \ref{sec4.1}.)   The first and last steps only concern matrix-vector multiplications and thus the computation cost is relatively low (by taking into account the fact that $m\ll n$ is not large in practice due to  the exponential order of accuracy in time).  The major computation is the second step for ${\bm s}_2$, but each diagonal block  can be computed  in parallel for all the $m$ blocks. In the above three steps, we need to be cautious about the roundoff error arising from diagonalizing $\CP$. Large roundoff error would seriously pollute the accuracy and according to the analysis in \cite{GH19,GW19a} the roundoff error is proportional to the condition number of the eigenvector matrix $V$, i.e., ${\rm Cond}_2(V)$.  For the proposed \pdd $\CP$, we show that ${\rm Cond}_2(V)$ is of moderate magnitude and only weakly grows as $m$ increases.

{Another contribution of this paper is a new strategy for  applying the \dg-based \pdd for   nonlinear problems (or linear problems with time-varying coefficient matrix).   For these problems, the widely used approach is the   average-based Kronecker  {product} approximation proposed in \cite{gander2017time}. This approach  works well  if the  variance  of the Jacobian matrices over the time points is small. But if the variance  is large,  it may result in slow convergence or even divergence for the preconditioned GMRES method.  Here, we use the  
\textit{nearest  Kronecker {product}  approximation} (NKPA) technique for handling Jacobian matrices and numerical results indicate that the resulting NKPA-based \pdd is much more effective than that obtained via the \textit{averaging} approach in 
  \cite{gander2017time}.}
 	
 		PinT algorithms for evolutionary   problems attract considerable attentions  in the last two decades \cite{gander201550}, mainly due to the advent of massively parallel processors that provide a potential to significantly speed up the traditional sequential time-stepping schemes.
	Given the sequential nature of the forward time evolution, the development of effective PinT algorithms is more challenging than the counterparts in space. 
	There are several different types of PinT algorithms in literature,   such as the parareal algorithm \cite{LMT01}, the multigrid reduction in time ({MGRiT})  algorithm \cite{FS14}, deferred correction methods \cite{christlieb2011implicit,ong2020deferred}, and the   \dg-based technique  \cite{MR08}.  
	 The mechanism of each algorithm varies greatly, which leads to significant difference in application scopes, convergence properties and   parallel efficiency.  In particular, the 
		 \dg-based technique   which is built upon diagonalizing   the time discretization matrix within the so-called  all-at-once system  shows promising speedup (see    numerical results in \cite{gander2020paradiag,  GW19b}).  As we will see in Section \ref{sec2}, such an all-at-once system arises naturally in the Sinc-Nystr\"{o}m methods and  therefore   we continue to investigate  such a technique in this paper.  {The \dgg technique was first proposed by Maday and R{\o}nquist in 2008 \cite{MR08} and then followed by many authors  \cite{GH19,Shu2018TOWARD,GW19b,MPW18,danieli2021spacetime,2011Parallel,2021A,LW20}.  (A summary of the \dg-based PinT algorithms  can be found in \cite{gander2020paradiag}.)  These previous work use the time-stepping method (e.g., the linear multistep methods or the Runge-Kutta method)  and the   time \diss  matrix is     a   lower triangular Toeplitz matrix (the all-at-once matrix $\CA$ is  of {\em block} version).  In this case, it is natural to define   the   \pdd $\CP$ as a block circulant matrix and  many good properties of the time \diss matrix, such as  the sparsity, Toeplitz structure and diagonal dominance, can be utilized for  the  spectral analysis of $\CP^{-1}\CA$. However, for the Sinc-Nystr\"{o}m method  the time \diss matrix      is a dense non-symmetric matrix  and there is no  clear structure for the all-at-once matrix $\CA$, which leads to essential difficulty for  constructing  an efficient \pdd and for  analyzing the spectrum of $\CP^{-1}\CA$. } 
 	
 	 	The {rest} of this paper is organized as follows. 
	In Section 2, we introduce the  Sinc-Nystr\"{o}m method
	for both linear and nonlinear initial-value ODEs, where the corresponding linear and nonlinear all-at-once systems are formulated.
	 In Section 3,  preconditioners for the heat equations and the wave equations are introduced,  where the spectrum of the preconditioned systems are carefully estimated. 
	In Section 4, we 
	study  the convergence performance of our proposed preconditioners for both parabolic and hyperbolic PDEs  and validate the spectrum analysis by several numerical experiments. We conclude this paper in   Section 5.

	\section{The Sinc-Nystr\"{o}m  method and the all-at-once system}\label{sec2}
	Following the notations used in \cite{stenger2012numerical,okayama2015theoretical}, in this section  we  briefly {revisit   the 
 Sinc-Nystr\"{o}m  method} for solving the linear and nonlinear initial-value ODEs.    The involved structured matrices for the all-at-once system are  given 
	for facilitating the later development and analysis of  the proposed  \pd.   
	
	\subsection{{The  Sinc-Nystr\"{o}m  method}} 
	For a given positive constant $d\in (0,\pi)$, we define a strip domain $\mathcal{D}_d$  in the complex plane  and a single-exponential conforming map $\phi(z)$
	$$
	\mathcal{D}_d:=\{z\in\IC: |{\rm Im}(z)|<d\}, ~\phi(z):=\ln\frac{z-a}{b-z}. 
	$$
	{The function $\phi(z)$ maps  a finite interval $(a,b)$ to $(-\infty,\infty)$}.   
	Define a domain $\mathcal{D}$  from  $\mathcal{D}_d$ via
	\[
	\mathcal{D}=\psi(\mathcal{D}_d):=\{z=\psi(\zeta): \zeta \in \mathcal{D}_d\},~\psi(z):=\phi^{-1}(z)=\frac{a+be^z}{1+e^z}. 
	\]
	{In this paper 
	we will only consider the case that $(a, b)$ is a bounded interval}, i.e.,   $(a,b)=(0,T)$, but 
 unbounded time intervals can  be addressed {as well}   by using  different conforming maps.
We denote by $\bm H(\mathcal{D})$ the family of analytic functions on $\mathcal{D}$
and for a given $h>0$ we define  the  Wiener function space 
\begin{equation}\label{eq4}
	\bm W(\pi/h):=\left\{f\in  \bm H(\IC): \int_{\IR} |f(t)|^2 dt<\infty\ {\rm and}\ |f(z)|\le C e^{\pi|z|/h}\right\},
\end{equation}
where $C>0$ is a constant.	
The  Sinc-Nystr\"{o}m  method is based on the   Sinc function on $\IR$
	\[ \Sinc(x)=\\
	\left\{
	\begin{array}{cl}
		\frac{\sin(\pi x)}{\pi x}, &x\ne 0,\\
		1, &x= 0.
	\end{array}
	\right.
	\]
	By shifting the Sinc function with a given $h>0$, we can define the set of Sinc basis functions 
	\[
	S[j,h](x):=\Sinc(x/h-j),\quad  j=0,\pm 1,\pm 2,\cdots,
	\]
	which forms a complete orthogonal sequence in the Winner function space $\bm W(\pi/h)$.
	Therefore, for any function $u\in \bm W(\pi/h)$ we have the    Sinc series expansion (also known as the Paley--Wiener theorem)
	\[
	u(x)={{\sum}}_{j=-\infty}^\infty u(jh)S[j,h](x),
	\]
	which {results in a practical numerical method after  truncation by choosing  ${M}$ and $h$  suitably} 
	\[
	u(x)\approx {{\sum}}_{j=-{M}}^{{M}} u(jh)S[j,h](x).
	\]
	{In practice, we can approximate any   function $f(t)$ defined on a finite interval $(a,b)$  
	through the  function  composition with the conformal map $\phi$} as follows
	\begin{equation}
		\label{f_h}
		f(t) \approx f_h(t):= {{\sum}}_{j=-{M}}^{{M}} f_j S[j,h]\circ\phi(t):={{\sum}}_{j=-{M}}^{{M}} f(t_j) S[j,h](\phi(t)), 
	\end{equation}
	where $f_j:=f(t_j)$ are the interpolation points at the $m=2{M}+1$ Sinc  time  points $t_j=\psi(jh ), j=-{M},\ldots,{M}.$
	Since the basis functions $S[j,h]\circ\phi(t)$ vanish at the end points $t=a$ and $t=b$, 
	the above Sinc approximation is not accurate near the end points if $f(a)\ne 0$ and/or $f(b)\ne 0$. To  handle $f(a)\ne 0$ and/or $f(b)\ne 0$,  the above   approximation  can be modified to 	\begin{equation}
		\label{mf_h}
		f(t) \approx \widehat f_h(t):=f_{-{M}} w_a(t)+f_N w_b(t)+ {{\sum}}_{j=-{M}}^{{M}} (f_j-f_{-{M}} w_a(t_j)-f_N w_b(t_j)) S[j,h]\circ\phi(t), 
	\end{equation}
	where  two auxiliary basis functions  $w_a(t):=(b-t)/(b-a)$ and $w_b(t):=(t-a)/(b-a)$
	are introduced to accommodate the possible nonzero end points.
{To use the above approximation for ODEs we also need the following integral form of  \eqref{f_h}}: 
	\begin{equation}\label{intf_h}
	\begin{split}
		\int_{a}^t f(s)ds \approx \int_{a}^t f_h(s)ds& = {\sum}_{j=-{M}}^{{M}} f_j \int_{a}^t S[j,h]\circ\phi(s)ds\\
		&= {\sum}_{j=-{M}}^{{M}} f_j \psi'(jh) J[j,h]\circ\phi(t),
		\end{split}
	\end{equation}
	where {$J[j,h](x):=h\left(\frac{1}{2}+\frac{1}{\pi}\int_0^{\pi (x/h-j)}\frac{\sin(\tau)}{\tau}d\tau\right)$}.

	We next revisit  exponential convergence results for the above two   approximations. To this end, we introduce the following function space 
	\[
	\bm H^{\infty}(\mathcal{D}):=\left\{f\in \bm H(\mathcal{D}):  {\rm sup}_{z\in\mathcal{D}}|f(z)|<\infty\right\}.
	\]
	For any positive constant $\alpha\in (0,1]$ and some constants $C_1$ and $C_2$, let 
\begin{equation*}
\begin{split}
&\bm L_{\alpha}(\mathcal{D}):=\left\{f\in \bm H^{\infty}(\mathcal{D}):  
	|f(z)|\le C_1 |(z-a)(b-z)|^{\alpha}
	\right\},\\
&\bm M_{\alpha}(\mathcal{D}):=\left\{f\in \bm H^{\infty}(\mathcal{D}): |f(z)-f(a)|\le C_2 |(z-a)|^{\alpha}\ {\rm and}\ |f(b)-f(z)|\le C_2 |(b-z)|^{\alpha}
	\right\}.
\end{split}
\end{equation*} 
	\begin{theorem}[\cite{stenger2012numerical}] \label{Thm2.1}
		Let $f\in\bm M_\alpha(\psi(\mathcal{D}_d))$ with $d\in (0,\pi)$ and ${M}$ be a positive integer.  By choosing  $h=\sqrt{\frac{\pi d}{\alpha {M}}}$,   there exists a constant $C$ (independent of ${M}$ and $h$) such that
		\[
		\max_{a\le t\le b} |f(t)-\widehat f_h(t)|\le C\sqrt{{M}}\exp(-\sqrt{\pi d \alpha {M}}).
		\]
	\end{theorem}
	\begin{theorem}[\cite{okayama2013error}] \label{Thm2.2}
		Let $f\in\bm L_\alpha(\psi(\mathcal{D}_d))$ with $d\in (0,\pi)$ and ${M}$ be a positive integer. By choosing  $h=\sqrt{\frac{\pi d}{\alpha {M}}}$,   there exists a constant $C$ (independent of ${M}$ and $h$) such that
		\[
		\max_{a\le t\le b} \left|\int_a^t f(s)ds-{{\sum}}_{j=-{M}}^{{M}} f_j \psi'(jh) J[j,h]\circ\phi(t)\right|\le C \exp(-\sqrt{\pi d \alpha {M}}).
		\]
	\end{theorem} 	
	
	\subsection{The  all-at-once system}
	We now introduce the Sinc-Nystr\"{o}m  method to linear and nonlinear ODE systems and the resulting all-at-once system. Efficient computation of such a  system  plays a central role in the practical applications of this method.  
	\subsubsection{Linear time-varying ODEs}
	We first consider the {following initial value ODEs}  
	\begin{equation} \label{ODEIVP}
		\begin{split}
			y'(t)&=K(t) y(t)+g(t),~  y(0)=r\in\IR^n,~t\in(0, T), 
		\end{split}
	\end{equation}
	where $y(t),g(t)\in\IR^n$ are vector functions and $K(t)\in\IR^{n\times n}$ is a time-dependent coefficient matrix. 
Such ODEs can also be derived from semi-discretized parabolic and hyperbolic PDEs.  {To apply the  Sinc-Nystr\"{o}m method, we} first rewrite \eqref{ODEIVP} into an  integral equation  	\[
	y(t)=r+\int_0^t \{K(s)y(s)+g(s)\}ds,~{t\in(0, T)}.
	\]
	According to  \eqref{intf_h}, we get the  Sinc-Nystr\"{o}m approximation of $y(t)$ as 
	\bea \label{Sinc-Col}
	y^h(t)= r+{\sum}_{j=-{M}}^{{M}} \{K(t_j)y^h(t_j)+g(t_j)\} \psi'(jh) J[j,h]\circ\phi(t),
	\eea
	which, by collocating at the same $m:=2{M}+1$ time points $\{t_l\}_{l=-M}^M$, leads to
	  
	\bea \label{Sinc-system}
	y^h(t_l)= r+{\sum}_{j=-{M}}^{{M}} \{K(t_j)y^h(t_j)+g(t_j)\} \psi'(jh) J[j,h]\circ\phi(t_l),\quad l=-{M},\cdots,{M}.
	\eea
	By definitions we have $\psi'(jh)=1/\phi'(t_j)$, $\phi(t_l)=\phi(\psi(lh))=lh$, and 
	\[
	J[j,h](lh)= h\left(\frac{1}{2}+ \int_0^{(l-j)} \frac{\sin(\pi t)}{\pi t} dt\right)
	=:h\sigma_{l-j}^{(-1)}.
	\]
	Define the $m\times m$ dense Toeplitz matrix 
	$$I^{(-1)}=\left[I^{(-1)}_{l,j} \right]:=\left[ \sigma^{(-1)}_{l-j}\right]_{l,j=1}^m=\left[\frac{1}{2}+ \int_0^{(l-j)} \frac{\sin(\pi t)}{\pi t} dt\right]_{l,j=1}^m,$$
	whose (complex) eigenvalues lie in the open right half plane \cite{han2014proof}. 
	For any given scalar function $g$, define the $m\times m$ diagonal matrix $\CD(g)={\rm diag}(g(t_{-{M}}),\cdots,g(t_{M}))$ over the $m$ time points. Let {$I_p\in\mathbb{R}^{p\times p}$ be} an identity matrix of size $p\times p$
	and $e_m=[1,1,\cdots,1]^\T \in \IR^m$ be a column vector of all ones.
	We use $(\cdot)^\T$ and $(\cdot)^*$ to denote the non-conjugate transpose and conjugate transpose, respectively.
	With the Kronecker product notations, the   Sinc-Nystr\"{o}m discretization scheme (\ref{Sinc-system})
	can be formulated into  an all-at-once linear system after suitable ordering the unknowns  
	\begin{align} \label{ODEIVPSinc}
		\CA\bm y_h:=\left(I_m\otimes I_n-(I^{(-1)}D\otimes I_n)\BK \right)\bm y_h=
		\bm b_h,
	\end{align}
	where $\bm y_h=[y(t_{-{M}}); \cdots;  y(t_{{M}})]\in\IR^{mn},\
	\bm g_h=[g(t_{-{M}});\ \cdots; \ g(t_{{M}})]\in\IR^{mn},\ 
	\bm f_h=e_m\otimes r \in\IR^{mn}
	$,    $D=h\CD(1/\phi')$ with $1/\phi'(t)=t(T-t)/T>0$, ${\bm b}_h=(I^{(-1)}D\otimes  I_n)\bm g_h+\bm f_h$ and $\BK$ is a block-diagonal matrix given by
	$$
	\BK={\rm blockdiag} \left(K(t_{-{M}}),\cdots,K(t_{{M}})) \right)\in \IR^{mn\times mn}.
	$$
%	{From \cite{stenger2012numerical}, we have}    $\|{I^{(-1)}D}\|_\infty\le 1.1T$ due to the fact $\sup_{x\in\IR}|J[j,h](x)|\le 1.1h$.
	In the simple case of constant coefficient matrix $K(t)=K$, there obviously holds
	$\BK=I_m\otimes K$ and hence  $(I^{(-1)}D\otimes I_n)\BK=(I^{(-1)}D\otimes I_n)(I_m\otimes K)=I^{(-1)}D\otimes K$, which reduces (\ref{ODEIVPSinc})  to
	\begin{align} \label{ODEIVPSincK}
		\CA\bm y_h:=\left(I_m\otimes I_n-I^{(-1)}D\otimes K \right)\bm y_h= \bm b_h.
	\end{align}
	
	Under certain assumptions on $K(t)$ and $g(t)$, it was shown {in} \cite{stenger2012numerical,okayama2018theoretical,hara2019error} that
	the linear {all-at-once system} (\ref{ODEIVPSinc}) with a sufficiently large ${M}$ is uniquely solvable and 
	the obtained Sinc approximation $\widehat y^h(t)$ in the form of (\ref{mf_h})   converges to $y(t)$ exponentially, i.e., 
	\[
 {\max_{0\le t\le T}\|y(t)-\widehat y^h(t)\|_{\infty}} =O\left(\sqrt{{M}}e^{-\sqrt{\pi d\alpha {M}}}\right).
	\]
 Although the exponential convergence of the above Sinc-Nystr\"{o}m discretization is well established,  {to the best of our knowledge
	the development of fast solvers for solving  the all-at-once linear systems (\ref{ODEIVPSinc}) and (\ref{ODEIVPSincK}) were not addressed in literature so far}.
	We note that an efficient solver for these all-at-once systems  is  crucial if the  ODE system  is very stiff {and/or the ODE system is of large scale}, such as the one  derived from  {semi-discretizing time-dependent PDEs}.  
	
		\subsubsection{Nonlinear ODEs}
	We next consider the nonlinear ODEs 
	\begin{align} \label{ODEIVP_NL}
		\begin{split}
			y'(t)&=q(t,y(t))+g(t), ~y(0)=r\in\IR^n,~{t\in(0, T)}, 
		\end{split}
	\end{align}
	where   $q(t,y(t))=[q_1(t,y(t)),q_2(t,y(t)),\cdots,q_n(t,y(t)))]^\T\in\IR^n$.
	The same Sinc-Nystr\"{o}m discretization of (\ref{ODEIVP_NL}) leads to a system of nonlinear equations 
	\bea \label{Sinc-system-NL}
	y^h(t_l)= r+{\sum}_{j=-{M}}^{{M}} \{q(t_j,y^h(t_j))+g(t_j)\} \psi'(jh) J[j,h]\circ\phi(t_l),\quad l=-{M},\cdots,{M}, 
	\eea
	which can be formulated into   the following all-at-once  form 
	\begin{align} \label{ODEIVPSinc-NL}
		\CF(\bm y_h):=(I_m\otimes I_n)\bm y_h-(I^{(-1)}D\otimes I_n)\mathsf{q} (\bm y_h)=
		(I^{(-1)}D\otimes  I_n)\bm g_h+\bm f_h=:\bm b_h,
	\end{align}
	with the nonlinear part 
	$\mathsf{q}(\bm y_h)=[q(t_{-{M}},y_{-{M}});\cdots;q(t_{{M}},y_{{M}})]\in \IR^{mn}$ (here   $y_l=y^h(t_l)$). 
	The Jacobian matrix of $\CF(y_h)$ reads
	\bea \label{JacF}
	\nabla \CF(y_h)=(I_m\otimes I_n)-(I^{(-1)}D\otimes I_n)\mathsf{Q}(\bm y_h),
	\eea
	where 
	\bea \label{Qh}
	\mathsf{Q}(\bm y_h):=\nabla_y \mathsf{q}(\bm y_h)={\rm blockdiag}(\nabla_y q(t_{-{M}},y_{-{M}}),\cdots,\nabla_y q(t_{{M}},y_{{M}})))\in \IR^{mn\times mn}
	\eea
	is {a} block-diagonal {matrix} 
	with $\nabla_y q$ being the Jacobian matrix of  $q$ with respect to $y$, given by
	\[
	\nabla_y q(t,y):=\bmt \frac{\partial q_1}{\partial y_1} &\frac{\partial q_1}{\partial y_2} &\cdots & \frac{\partial q_1}{\partial y_n} \\
	\frac{\partial q_2}{\partial y_1} &\frac{\partial q_2}{\partial y_2} &\cdots & \frac{\partial q_2}{\partial y_n}\\
	\vdots\\
	\frac{\partial q_n}{\partial y_1} &\frac{\partial q_n}{\partial y_2} &\cdots & \frac{\partial q_n}{\partial y_n}
	\emt\in\IR^{n\times n}.
	\]
	{Applying Newton's iteration to  (\ref{ODEIVPSinc-NL})  leads to}
	\bea \label{Newton}
	\bm y_h^{(k+1)}=\bm y_h^{(k)}-\left[\nabla\CF(y_h^{(k)})\right]^{-1}\left(\CF(\bm y_h^{(k)})-\bm b_h\right),\quad k=0,1,2,\cdots
	\eea
	{where  $\bm y_h^{(0)}$ is the initial guess}. 
	We see that  the Jacobian  matrix $\nabla\CF(y_h^{(k)})$ in (\ref{Newton})  has  the same structure {as   (\ref{ODEIVPSinc})} and therefore a \pdd  for (\ref{ODEIVPSinc}) is   also applicable to   (\ref{Newton}) as well. 	For convergence of  the above Newton iteration,
	a variant of the well-known Newton-Kantorovich theorem is given in  \cite[p. 344, Theorem 6.4.4]{stenger2012numerical}.  		In general, the Newton iteration achieves only local convergence within a short time window and 
		to handle   a much longer time interval we can first split the whole time interval into several subintervals  and  then apply the Newton iterations to these subintervals  one after another. 
	
\section{{The}   \pdd  and {the} spectrum  analysis}\label{sec4}
	In this section,  we  {first propose} a  PinT \pdd $\CP$ for solving the  all-at-once system (\ref{ODEIVPSinc}) and {then we  give a} spectral analysis for the preconditioned matrix $\CP^{-1}\CA$.      
		We start by discussing the simple constant coefficient case (\ref{ODEIVPSincK}), where the {all-at-once} matrix is 
	\begin{align} \label{constantK}
		\CA = I_m\otimes I_n-I^{(-1)}D\otimes K.
	\end{align}
	{The \pdd for $\CA$ is different for the case $\sigma(K)\subset\mathbb{R}$ and $\sigma(K)\subset\Ii\mathbb{R}$, where $\sigma(K)$ denotes the spectrum of $K$. We note that these are two representative cases: the first case represents that the differential equation is dissipative while the second case corresponds to wave propagation problems (e.g., $K$ is the discrete matrix of a wave equation). }
\subsection{{The    \pdd for  the case $\sigma(K)\subset\mathbb{R}$}}\label{sec4.1}
	In view of the special Toeplitz structure of $I^{(-1)}$ in $\CA$, we propose the following \pd 
	\bea \label{pre_P}
	\CP= I_m\otimes I_n- S D\otimes K,
	\eea
	where the Toeplitz matrix $I^{(-1)}$  is approximated by its skew-symmetric part  \cite{ng2003circulant}:
	$$
	S=\frac{I^{(-1)}-(I^{(-1)})^\T}{2}.
	$$ 
	 {A routine calculation shows}   that $S$ is skew-circulant and skew-symmetric (i.e. $S^\T=-S$).
	{Moreover, it holds} 
	\[  S_{k,l}=\frac{1}{2}\int_{l-k}^{k-l} \frac{\sin (\pi t)}{\pi t} dt= \int_{0}^{k-l} \frac{\sin (\pi t)}{\pi t} dt= \sigma^{(-1)}_{k-l}-\frac{1}{2}, 
	\]
	and hence $S$ is a rank-one perturbation of $I^{(-1)}=\left[ \sigma^{(-1)}_{k-j}\right]_{k,j=1}^m$ according to
\begin{equation}\label{eq3.3}
	S=I^{(-1)}-\frac{1}{2} e_m e_m^\T \text{ with } e_m:=[1,1,\cdots,1]^\T.
\end{equation} 
So $\CP$ is a rank-$n$ perturbation of $\CA$ and it is anticipated to be an effective preconditioner of $\CA$.
	In \cite{abu2003properties}, it was shown that $S$ is  unitrary diagonalizable and all the eigenvalues of $S$ are simple. Furthermore, in \cite[Theorem 2.1]{han2014proof} it was shown that $(S+\epsilon e_m e_m^\T)$ (including $I^{(-1)}$ as a special case) is nonsingular for any $\epsilon>0$ and has all its  eigenvalues {lie} in the open right half-plane.  
%	In \cite{han2014proof} it was shown that any eigenvector $v$ of $I^{(-1)}$ satisfies $e_m^\T v\ne 0$.

\subsubsection{{Implementation details}}\label{sec3.1.1}	

	{Since} the  skew-symmetric matrix $S$ is diagonalizable and $SD$ is similar to  the skew-symmetric matrix $D^{\frac{1}{2}}SD^{\frac{1}{2}}$,  the matrix $SD$ is also diagonalizable.
	Let $SD=V \Sigma V^{-1}$ be its diagonalization (or eigen-decomposition).  Then we can factorize $\CP$ {as}
	\[
	\CP=(V\otimes I_n)\left(I_m\otimes I_n-\Sigma\otimes K\right)(V^{-1}\otimes I_n).
	\]
	Hence, for any vector ${\bm r}$  the preconditioning step ${\bm s}=\CP^{-1}{\bm r}$  can be computed by three steps:
	\begin{equation}\label{3step2}
		\begin{split}
			&\text{Step-(i)} ~~{\bm s}_1=\texttt{mat}({\bm r})(V^{-1})^{\T} \in\IR^{n\times m},\\
			&\text{Step-(ii)} ~~{\bm s}_{2}(:,j)=(I_n-\lambda_{j}K)^{-1} {\bm s}_{1}(:,j),\quad ~j=1,2,\dots,m,\\
			&\text{Step-(iii)} ~~{\bm s}=\texttt{vec}({\bm s}_2V^\T)\in\IR^{mn},\\
		\end{split}
	\end{equation} 
	where $\Sigma=\text{diag}(\lambda_{1},\dots,\lambda_{m})$ and ${\bm s}_{1,2}(:,j)$ denotes the $j$-th column of ${\bm s}_{1,2}$. In \eqref{3step2},  we have used the reshaping operations:  matrix-to-vector $\texttt{vec}$ and vector-to-matrix $\texttt{mat}$.  
	Clearly, the $m$   independent  linear systems in Step-(ii) can be computed in parallel.
	\begin{remark}
For the \pdd $\CP$ in \eqref{pre_P} there is a more convenient implementation of $\CP^{-1}{\bm r}$.  We can 
  factorize   $\CP$ as 
	$
	\CP=(D^{-\frac{1}{2}}\otimes I_n)\left(I_m\otimes I_n-D^{\frac{1}{2}}SD^{\frac{1}{2}}\otimes K \right)(D^{\frac{1}{2}}\otimes I_n)
	$ with    $D^{\frac{1}{2}}SD^{\frac{1}{2}}$ being   skew-symmetric  since $$
	\left(D^{\frac{1}{2}}SD^{\frac{1}{2}}\right)^\T=D^{\frac{1}{2}}S^\T D^{\frac{1}{2}}=-\left(D^{\frac{1}{2}}SD^{\frac{1}{2}}\right). 
	$$    
	This implies that  $D^{\frac{1}{2}}SD^{\frac{1}{2}}$ is a normal matrix and it {is} unitary diagonalizable: $D^{\frac{1}{2}}SD^{\frac{1}{2}}=W\Sigma W^*$ 
	with a unitary matrix $W$.  Therefore, we can replace $V^{-1}$ in \eqref{3step2} by $W^*D^{\frac{1}{2}}$, i.e., there is no need to invert  the eigenvector matrix.  	\end{remark}

One may wonder  \textit{why we do not directly  factorize $I^{(-1)}D$  and then solve the  all-at-once system $\CA{\bm y}_h={\bm b}_h$  by  the above \dgg procedure?}    This is  indeed the most convenient approach but unfortunately it does not work due to large roundoff errors arising from diagonalization of $I^{(-1)}D$.  
	%The reason is due to very large roundoff errors arising from diagonalizing any square matrix $A=VDV^{-1}$.   (Because of floating  point operations  it only holds   $A\approx V^{-1}DV$  in practical computation and the difference between $A$ and $V^{-1}DV$ causes roundoff errors.)  
	According to the analysis in \cite{GH19,GW19a},  the roundoff errors for the diagonalization procedure \eqref{3step2} is proportional to the  condition number of the eigenvector matrix $V$ (denoted by   ${\rm Cond}_2(V)$). A very large ${\rm Cond}_2(V)$  leads to large roundoff error that will seriously pollute the accuracy of obtained numerical solution.   Let $U$ and $V$ be respectively the eigenvector matrix   of  $I^{(-1)}D$ and  $SD$.  	 In Figure \ref{fig3.1},  we compare  the condition number for $U$ and $V$ as a function of system size $M$.  Here, we use the   \texttt{eig} function in MATLAB  for both $U$ and $V$.  Clearly,  ${\rm Cond}_2(V)$ is much smaller than ${\rm Cond}_2(U)$ and the former seems increases only linearly.  The condition number  ${\rm Cond}_2(U)$   grows exponentially  as ${M}$ increases.   Our numerical simulations indicate  that  if we directly solve the all-at-once system $\CA$ by utilizing the diagonalization $I^{(-1)}D=U\Psi U^{-1}$, the unavoidable large roundoff error seriously pollutes the solution accuracy for $M\ge 64$ (see the last column of Table \ref{T3wave} in Section  \ref{sec5}).
		\begin{figure}[H]
		\centering 
		\includegraphics[width=1\textwidth]{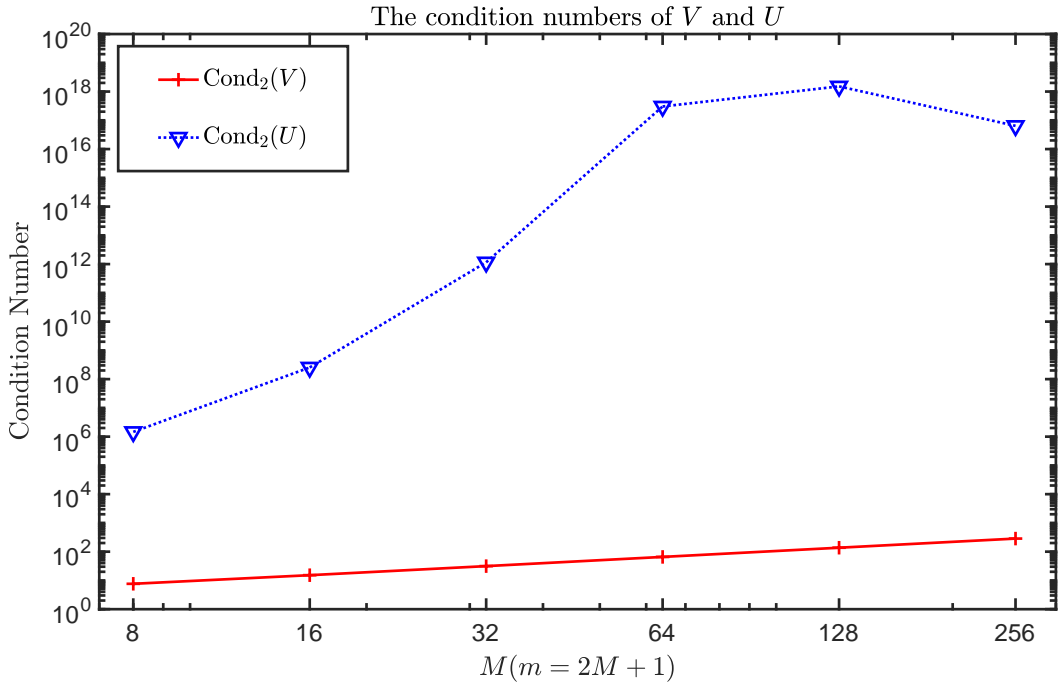}  
		\caption{The growth of ${\rm Cond}_2(U)$ and ${\rm Cond}_2(V)$ with $U$ and $V$ being the eigenvector matrix   of  $I^{(-1)}D$ and  $SD$, respectively.  The huge condition number (over $10^{15}$ for $M\ge 64$) of $U$ implies that we can not directly invert $\CA^{-1}{\bm b}_h$ by the diagonalization  technique, because the roundoff errors will seriously pollute the solution accuracy.  The mildly increasing condition number of $V$ will not contaminate the approximation accuracy.} \label{fig3.1}
	\end{figure}

{The following lemma presents an estimate  of ${\rm Cond}_2(V)$, but   is seems rather pessimistic compared to the numerical  result  shown in   Figure \ref{fig3.1}.  We mention  that  the eigenvector matrix $V$ is not unique because   $V\Phi$ is also an eigenvector matrix with any nonsingular diagonal matrix $\Phi$. Hence it entirely impossible to improve the following estimate with some suitable scaling matrix $\Phi$. We however do not further pursue this goal in the current paper.} 
  \begin{lemma} \label{condV}
	Let $V\Sigma V^{-1}$ be a \dgg of $SD$. It holds 
	${\rm Cond}_2(V)=O(e^{\sqrt{{M}}})$.
	\end{lemma}
	\begin{proof}
	It follows from  $SD=V \Sigma V^{-1}$ and $D^{\frac{1}{2}}SD^{\frac{1}{2}}=W\Sigma W^*$ that $V=D^{-\frac{1}{2}}W$. Since $W$ is unitary with $W^{-1}=W^*$, there holds
	\[
	{\rm Cond}_2(V)=\|V\|_2 \|V^{-1}\|_2=\|D^{-\frac{1}{2}}W\|_2 \|W^*D^{\frac{1}{2}}\|_2=\|D^{-\frac{1}{2}}\|_2 \|D^{\frac{1}{2}}\|_2={\rm Cond}_2(D^{\frac{1}{2}}).
	\]
	{Recall that  $t_{-{M}}=\psi(-Mh)=\frac{Te^{-Mh}}{1+e^{-Mh}}$ and $h=\sqrt{\frac{\pi d}{\alpha M}}$} (chosen in Theorem \ref{Thm2.1}),  it holds that
	\[
	{\rm Cond}_2(D^{\frac{1}{2}})=\sqrt{{\rm Cond}_2(D)}=\sqrt{\frac{T^2/4}{t_{-M}(T-t_{-M})}}={e^{Mh/2}{(1+e^{-Mh})}/2}=O\left(e^{\sqrt{{M}}}\right).
	\] 
	Hence, ${\rm Cond}_2(V)={\rm Cond}_2(D^{\frac{1}{2}})=O\left(e^{\sqrt{{M}}}\right)$. 
	\end{proof} 

	\subsubsection{Spectrum  analysis  of $\CP^{-1}\CA$}\label{sec3.1.2}
    	We  now analyze the eigenvalues of $\CP^{-1}\CA$ for the case with constant coefficient. 	In general, a clustered spectrum of $\CP^{-1}\mathcal{A}$
	indicates the effectiveness of the preconditioner $\CP$ in practice,
	although the rigorous convergence rate of the preconditioned GMRES is not conclusively determined by the spectrum alone (see e.g.  \cite{greenbaum1996any,meurant2015role,wathen2015preconditioning}), especially for  {non-normal systems}.

	For simplicity, we assume that $K$ can be diagonalized as $K=Q \Gamma Q^{-1}$, which is {often} the case if $K$ {is the discrete matrix of self-adjoint elliptic  operator, e.g., the Laplacian}. With $K=Q \Gamma Q^{-1}$ we get
	the following factorization of $\CA$ and $\CP$ 
	\begin{equation*}
		\begin{split}
			&\mathcal{A}=(I_m\otimes Q)(I_m\otimes I_n- I^{(-1)} D\otimes \Gamma)(I_m\otimes Q^{-1}),\\
			&\mathcal{P}=(I_m\otimes Q)(I_m\otimes I_n- SD\otimes \Gamma)(I_m\otimes Q^{-1}).
		\end{split}
	\end{equation*} 	 	 
		The following lemma   will be used to  estimate  the spectrum of  $\CP^{-1}\mathcal{A}$. 
		 	 	\begin{lemma} \label{lemma_zmu}
		 	 	Let $z(\mu):=e_m^\T\left(\mu^{-1}D^{-1}+I^{(-1)}\right)^{-1} e_m$.  It holds $z(\mu)\in [0,2)$ for $\mu\geq0$.  
		 	 \end{lemma} 
%{{\em Note}. 	In \cite[Theorem 2.1]{han2014proof} it was shown that $I^{(-1)}=S+\frac{1}{2} e_m e_m^\T$  is nonsingular  and has all its  eigenvalues in the open right half-plane,
%	which implies $\mu^{-1}D^{-1}+I^{(-1)}$ with $\mu\ne0$ and $\Re(\mu)\ge 0$ is also nonsingular since $D$ is diagonal with positive entries.	}		 
		 	 \begin{proof}
		 	 	{For $\mu=0$, the   result   holds   trivially  since}  $z(0)=0$. {Hence,   we will only} discuss the case   $\mu> 0$. 		
				 For  this case we first prove that  $z(\mu)$ is well-defined, i.e., $\mu^{-1}D^{-1}+I^{(-1)}$ is nonsingular. To this end, we let  $(\lambda,\xi)$ with $\|\xi\|_2=1$ be any (complex) eigenpair of  $\mu^{-1}D^{-1}+I^{(-1)}$, that is $\lambda \xi=(\mu^{-1}D^{-1}+I^{(-1)}) \xi=(\mu^{-1}D^{-1}+S+\frac{1}{2} e_me_m^\T)\xi$. Since 	 $\|\xi\|_2=1$, it holds 	 	
				 \[
		 	 	 \lambda=\lambda \|\xi\|_2^2=\lambda \xi^*\xi=\mu^{-1}\xi^*D^{-1}\xi +\xi^*S\xi+\frac{1}{2}  \xi^*e_me_m^\T \xi=
				 \mu^{-1}\xi^*D^{-1}\xi+\xi^*S\xi+\frac{1}{2}  \|e_m^\T\xi\|_2^2. 
		 	 	\]
By noticing that  $\xi^*D^{-1}\xi>0$  (due to {the fact that} $D=h\CD(1/\phi')$ is diagonal with positive entries) and  $\xi^*S\xi$ is purely imaginary (due to $S^*=S^\T=-S$), we have 
		 	 	$$
		 	 	 \Re(\lambda)= \mu^{-1}\xi^*D^{-1}\xi+\frac{1}{2}  \|e_m^\T\xi\|_2^2>0. 
		 	 	$$  
 Hence, $\mu^{-1}D^{-1}+I^{(-1)}$ is indeed nonsingular, i.e., $z(\mu)$ is well-defined. 			
 					
					Let $\Phi_{\mu}:=\mu^{-1}D^{-1}+S$. Then  following the proof for	 $\mu^{-1}D^{-1}+I^{(-1)}$ we can show that $\Phi_{\mu}$ is nonsingular for $\mu>0$ as well. 
		 	 	By using the Sherman-Worrison-Woodbury formula \cite{golub2012matrix}, 		 	 		\begin{equation} \label{SWWformula}
		 	 	\left(\mu^{-1}D^{-1}+I^{(-1)}\right)^{-1}
		 	 	=\left(\Phi_{\mu}+\frac{1}{2} e_me_m^\T\right)^{-1}
		 	 	=\Phi_{\mu}^{-1}-\frac{\Phi_{\mu}^{-1} e_m e_m^\T \Phi_{\mu}^{-1} }{2+e_m^\T \Phi_{\mu}^{-1} e_m},
		 	\end{equation}
		 	 	which implies
		 	 	\begin{equation} \label{zmuformula}
		 	 		z(\mu)=e_m^\T \Phi_{\mu}^{-1} e_m-\frac{(e_m^\T \Phi_{\mu}^{-1} e_m)^2}{2+e_m^\T \Phi_{\mu}^{-1} e_m}=\frac{2 e_m^\T \Phi_{\mu}^{-1} e_m}{2+e_m^\T \Phi_{\mu}^{-1} e_m} 
		 	 		=2-\frac{4}{2+e_m^\T \Phi_{\mu}^{-1} e_m}.
		 	 	\end{equation}
		 	 	
		 	  Let $\gamma(\mu):=e_m^\T \Phi_{\mu}^{-1} e_m$ and $v:=\Phi_{\mu}^{-1} e_m$. {Then it holds}  $v\ne 0$ and 	$e_m=\Phi_{\mu} v$, 
		 	 	which, using the fact that $S$ is skew-symmetric with  $v^\T S v=0$, leads to
		 	 	$$v^\T e_m=v^\T\left(\mu^{-1}D^{-1}+S\right) v=\mu^{-1} v^\T D^{-1} v.$$
		 	 	This together with  the fact that   $D=h\CD(1/\phi')$ is diagonal with positive entries and
		 	 	$$
		 	 	\gamma(\mu)=e_m^\T \Phi_{\mu}^{-1} e_m=e_m^\T v=(e_m^\T v)^\T=v^\T e_m
		 	 	$$
		 	 	gives  
		 	 	$ \gamma(\mu) =\mu^{-1} v^\T D^{-1} v>0$ for $\mu>0$.  
		 	 	In view of (\ref{zmuformula}) , we obtained the desired result
		 	 	$z(\mu)=2-\frac{4}{2+\gamma(\mu)}\subset [0,2)$.
		 	 \end{proof}
	 	 
		 	 As a numerical illustration of Lemma \ref{lemma_zmu}, {in}  Figure \ref{fig3.2} {we} plot the function $z(\mu)$ with different ${M}$ for   $\mu>0$. We see that $z(\mu)\to 2$ as $\mu\to\infty$ for a fixed $M$,  but how fast $z(\mu)$ approaches $2$ seems to highly depend on $M$. 
		 	 \begin{figure}[H]
		 	 	\centering 
		 	 	\includegraphics[height=2.6in,width=5.5in]{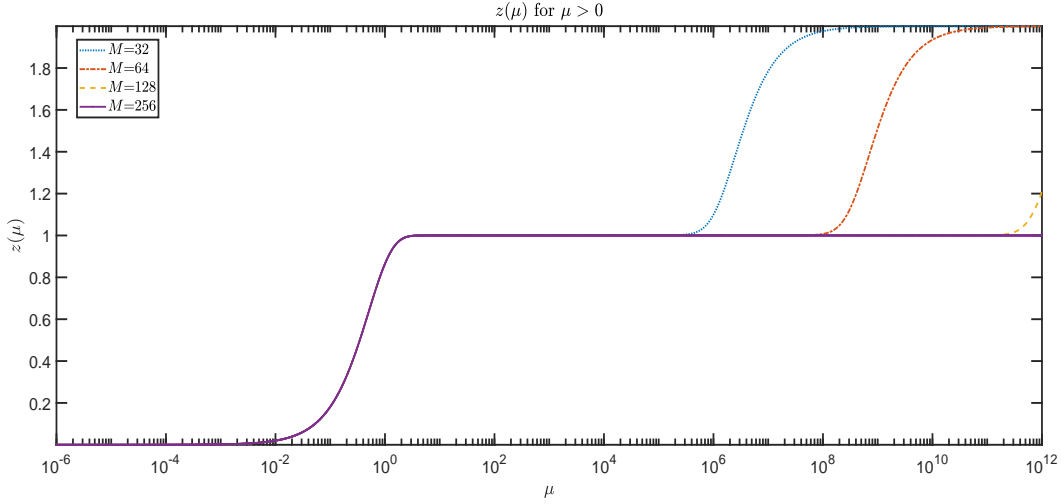}   		 	 	\caption{The function $z(\mu)$ for $\mu\in[10^{-6}, 10^{12}]$  with different ${M}$. Here we take $T=2, d=\pi/2$ and  $\alpha=1$. 
		 	 	} \label{fig3.2}
		 	 \end{figure}

		\begin{theorem} \label{ThmPA}
Suppose $K$ is diagonalizable with  negative spectrum $\sigma(K)\subset (-\infty, 0)$. Then,  $\CP^{-1}\mathcal{A}$ has   $n(m-1)$ unity eigenvalues and $n$ non-unity eigenvalues. {Moreover, it holds} 
		\[
		\sigma(\CP^{-1}\mathcal{A})\subset [1,  \infty).
		\]		 
	\end{theorem} 
	\begin{proof}
		Let  $-\mu\in\sigma(K)$ with $\mu>0$ be an arbitrary eigenvalue of $K$. Then, it is clear that
		$$
		\sigma(\CP^{-1}\mathcal{A})=\sigma\left((I_m\otimes I_n- SD\otimes \Gamma)^{-1} (I_m\otimes I_n- I^{(-1)} D\otimes \Gamma)\right)={\bigcup}_{-\mu\in\sigma(K)}\sigma( P^{-1}_\mu A_\mu), 
		$$
		where 
		$P_\mu=I_m+ \mu S D=\mu \Phi_{\mu}D$, $A_\mu=I_m+ \mu I^{(-1)} D=\mu(\Phi_{\mu}+\frac{1}{2} e_me_m^\T)D$ and $\Phi_{\mu}= \mu^{-1}D^{-1}+S$.  (In the proof of Lemma \ref{lemma_zmu} we have already proved that $\Phi_{\mu}$ is nonsingular and thus $P_{\mu}$ is nonsingular as well.)  
		Since $S=I^{(-1)}-\frac{1}{2} e_me_m^\T$, by the Sherman-Worrison-Woodbury formula \cite{golub2012matrix}  we have 
		\begin{equation*}
		\begin{split}
		P_\mu^{-1} A_\mu=\left( A_\mu-\frac{\mu}{2}e_m e_m^\T D\right)^{-1}A_\mu
		 =I_m- A_\mu^{-1}e_m\left(-\frac{2}{\mu}+e_m^\T D A_\mu^{-1}e_m \right)^{-1} e_m^\T D,
		\end{split}
		\end{equation*}
		which is a rank-one   perturbation of the identity matrix. 
		Hence $P_\mu^{-1} A_\mu$ has $(m-1)$ unity eigenvalues and the remaining only  one non-unity eigenvalue is given by
	\begin{equation}\label{eq3.5}
		\lambda_{\max}(P_\mu^{-1} A_\mu)=1-\frac{ e_m^\T D A_\mu^{-1}e_m}{-\frac{2}{\mu}+e_m^\T D A_\mu^{-1}e_m}
		=\frac{2}{2-\mu e_m^\T D A_\mu^{-1}e_m}=\frac{2}{2-z(\mu)},
		\end{equation}
		where we have used the   fact
		\begin{equation} \label{zmu}
		\begin{split}
		\mu e_m^\T D A_\mu^{-1}e_m=
		\mu e_m^\T D(I_m+\mu I^{(-1)} D)^{-1}e_m=e_m^\T {\left(\mu^{-1}D^{-1}+I^{(-1)}\right)^{-1} e_m} =z(\mu).
		\end{split}
		\end{equation}
		From Lemma \ref{lemma_zmu},  we have $z(\mu)\in [0,2)$ for $\mu> 0$ and hence $\lambda_{\max}(P_\mu^{-1} A_\mu)=\frac{2}{2-z(\mu)}\ge 1$, which completes the proof. 
	\end{proof}
	 
		In Figure  \ref{fig1eigplot}, {we plot}   the computed {eigenvalues} of $\CA$ and $\CP^{-1}\CA$  {for the} 1D heat equation (cf. Example 1 in Section \ref{sec5}).
	From Figure \ref{fig1eigplot} we see that the eigenvalues of $\CP^{-1}\CA$ are real (if neglecting the roundoff errors) and highly clustered around 1 (within a bounded interval $[1.75, 2]$). {Our numerical results in Table \ref{T2heat_2D}  show that    the preconditioned GMRES   converges in only a few iterations}.
	 		  
	\begin{figure}[htp!]
		\centering		 
		\includegraphics[height=2.57in,width=5.65in]{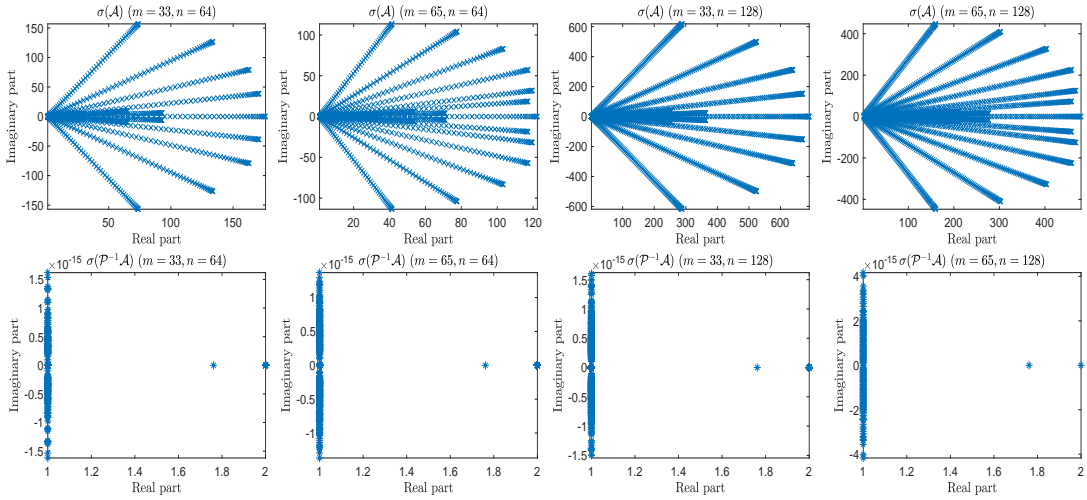} 
		\caption{The eigenvalues  of $\CA$ and $\CP^{-1}\CA$   {for Example 2 in Section \ref{sec5}, i.e., 1D heat equation with $n=64,128$ and $m=33,65$}. {Note that $m(n-1)$ eigenvalues  of $\CP^{-1}\CA$ are one and all eigenvalues are real.}} \label{fig1eigplot}
	\end{figure} 
 
	\begin{remark}  
		From \eqref{eq3.5} the maximum of $z(\mu)\in [0,2)$  controls the upper bound of the eigenvalues of $P_{\mu}^{-1}A_{\mu}$.
	From  Figure \ref{fig3.2},  we know that  $z(\mu)\to 2$  as $\mu$ gets larger and hence $\lambda_{\max}(P_{\mu}^{-1}A_{\mu})$ may become large as well.
		Numerically, due to the highly clustering of the eigenvalues of $\CP^{-1}\CA$ with real negative spectrum $\sigma(K)$,  
		a few very large eigenvalues do not seem to cause obvious degeneration  of convergence rate  for the preconditioned GMRES method.
	\end{remark} 
\subsection{{The    \pdd for  $\sigma(K)\subset\Ii\mathbb{R}$}}\label{sec3.2}
{
{We now consider the wave propagation problems, i.e.,   $\sigma(K)\subset\Ii\mathbb{R}$. In this case, the matrix  $\Phi_{\mu}$  defined in Lemma \ref{lemma_zmu} (i.e., $\Phi_{\mu}:=\mu^{-1}D^{-1}+S$) could be singular  for some special $\mu$ (the singularity of $\Phi_{\mu}$ implies singularity of $P_\mu=\mu \Phi_{\mu}D$ or equivalently $\CP$).  In fact, for any  purely imaginary eigenvalue  $\Ii\tilde{\mu}$ of $S^{\frac{1}{2}}DS^{\frac{1}{2}}$ we can choose   $\mu=\pm \Ii\tilde{\mu}^{-1}$ such that   the matrix $\Phi_{\mu}$ is singular.    Hence, generally speaking the \pdd $\CP$ proposed in Section \ref{sec4.1} is not applicable to wave propagation problems.

The above discussion  motivates us to propose and study an improved \pdd which actually works very well  for both the parabolic and  wave equations. The new \pdd is a generalized version of $\CP$   parameterized  by  a small parameter $\omega\in (0,1)$ which is used to  control the norm (or magnitude) of the rank-one perturbation term in constructing  $\CP$}:
\begin{equation} \label{pre_P_omega}
	\CP(\omega)= I_m\otimes I_n- (S(\omega) D)\otimes K,
	\end{equation}
	where $S(\omega)$ is defined as a damped rank-one perturbation of $I^{(-1)}$:
	\begin{equation}\label{S-omega}
		S({\omega})=I^{(-1)}-\frac{\omega}{2} e_m e_m^\T=S+\frac{1-\omega}{2} e_m e_m^\T, \quad \omega\in (0,1).
	\end{equation}
	When $\omega$ is small (e.g. $\omega=0.01$) the preconditioner $\CP(\omega)$ is expected to perform better than $\CP=\CP(1)$ in view of $\lim_{\omega\to 0} \CP(\omega)=\CA$.
 	Suppose  $S(\omega) D$ is diagonalizable with $S(\omega) D=V_\omega\Sigma_\omega V^{-1}_\omega$, we expect that the condition number ${\rm Cond}_2(V_\omega)$     ranges from ${\rm Cond}_2(V)$ to ${\rm Cond}_2(U)$, where  {$U$ is the eigenvector matrix of $I^{(-1)}D$. With the \dgg of $S(\omega)D$, the computation   of    $\CP^{-1}(\omega){\bm r}$  is the same as  the 3-step procedure \eqref{3step2} and we omit the presentation.    }    
	
The growth of ${\rm Cond}_2(V_\omega)$ as $M$ increases  is illustrated in Figure \ref{CondV_omega}, where  
${\rm Cond}_2(V_\omega)$ seems to be proportional to $\frac{1}{\omega}{\rm Cond}_2(V)$. {(Unlike the \pdd $\CP$ proposed in Section \ref{sec4.1}, the analysis of ${\rm Cond}_2(V_\omega)$|even though a rough estimate as given by Lemma \ref{condV} for $\CP$,  is  extremely difficult.)  This implies that  the roundoff error arising from the \dgg procedure would be well controlled by   choosing a moderate  $\omega$.  The remained question is how the parameter $\omega$ influences the spectrum of the preconditioned matrix $\CP^{-1}(\omega)\CA$}. 
	\begin{figure}[htp!]
		\centering  
		\includegraphics[width=1\textwidth]{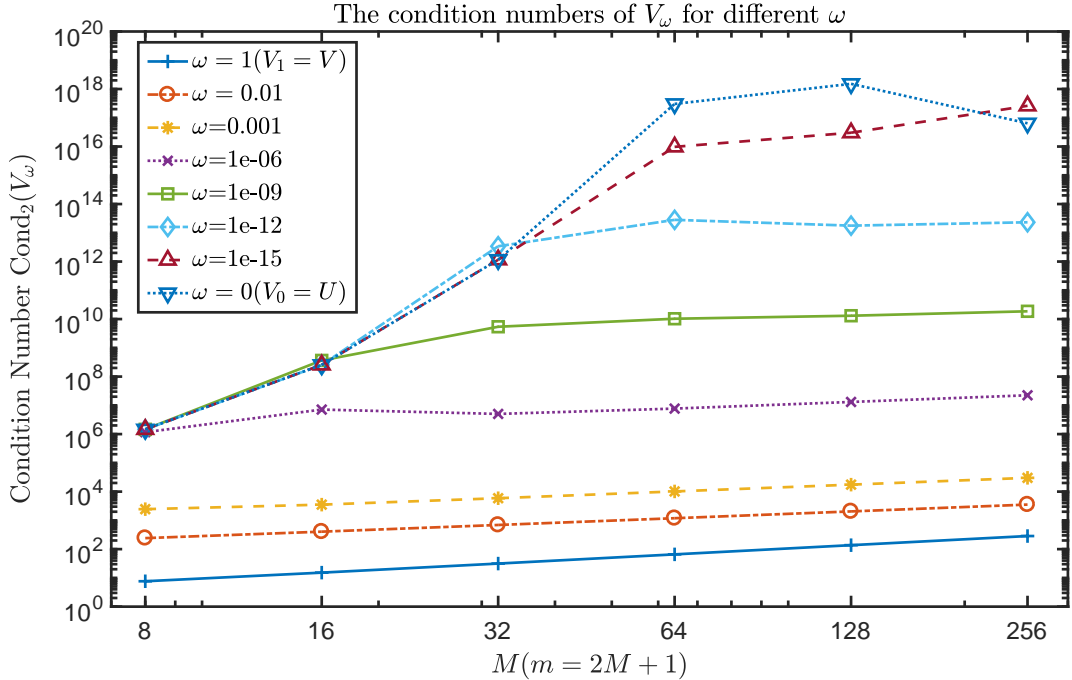}  
		\caption{The growth of the condition number ${\rm Cond}_2(V_\omega)$ as a function of $M$ for different $\omega$ values.} \label{CondV_omega}
	\end{figure} 
 
\begin{theorem}[{the case $\sigma(K)\subset\mathbb{R}^-$}] \label{ThmPAheat_omega}
	{Let} $K$  be a diagonalizable  matrix with real negative eigenvalues.  Then $\CP^{-1}(\omega)\mathcal{A}$ with $\omega\in (0,1)$ has only $n$ non-unity eigenvalues and   
		\[
	\sigma(\CP^{-1}(\omega)\mathcal{A})\in \left[1,\frac{1}{1-\omega}\right). 
	\]
\end{theorem} 
\begin{proof} 
	Following the proof arguments of Theorem \ref{ThmPA},
	it is clear that
	$$
	\sigma(\CP^{-1}(\omega)\mathcal{A})= {\bigcup}_{-\mu\in\sigma(K)}\sigma( P^{-1}_\mu(\omega) A_\mu), 
	$$
	where 
	$P_\mu(\omega)=I_m+ \mu S(\omega) D=\mu(\mu^{-1}D^{-1}+  S(\omega)) D$  is nonsingular for $\mu>0$ since $S(\omega)$ is nonsingular with all eigenvalues located in the open right half-plane.  
	Since $S(\omega)=I^{(-1)}-\frac{\omega}{2} e_me_m^\T$, by the Sherman-Worrison-Woodbury formula \cite{golub2012matrix}  we have 
	\begin{equation*}
		\begin{split}
			P_\mu^{-1}(\omega) A_\mu=\left( A_\mu-\frac{\mu\omega}{2}e_m e_m^\T D\right)^{-1}A_\mu
			=I_m- A_\mu^{-1}e_m\left(-\frac{2}{\mu\omega}+e_m^\T D A_\mu^{-1}e_m \right)^{-1} e_m^\T D,
		\end{split}
	\end{equation*}
	which implies {that} $P_\mu^{-1}(\omega) A_\mu$ has $(m-1)$ unity eigenvalues and only one non-unity eigenvalue 
	\begin{equation}\label{eq3.5_omega2}
		\lambda(P_\mu^{-1}(\omega) A_\mu)=1-\frac{ e_m^\T D A_\mu^{-1}e_m}{-\frac{2}{\mu\omega}+e_m^\T D A_\mu^{-1}e_m}
		=\frac{2}{2-\omega\mu e_m^\T D A_\mu^{-1}e_m}=:\frac{2}{2-\omega z(\mu)},
	\end{equation}
	where $z(\mu)$ is the same function defined by Lemma \ref{lemma_zmu} satisfying  $z(\mu)\in [0,2)$ for $\mu>0$.
	Hence 
	\bean
	\lambda (P_\mu^{-1}(\omega) A_\mu)=\frac{2}{2-\omega z(\mu)}\in \left[1,\frac{1}{1-\omega}\right),
	\eean
	where together with other unity eigenvalues completes the proof.
\end{proof}

For the case that the eigenvalues of $K$ are purely imaginary, we have the following uniform bounds for the spectrum of $\CP^{-1}(\omega)\mathcal{A}$.
	\begin{theorem}[{the case $\sigma(K)\subset\Ii\mathbb{R}$}] \label{ThmPAwave_omega}
		{Let} $K$  be a diagonalizable  matrix with purely imaginary spectrum $\sigma(K)$. {Then} $\CP^{-1}(\omega)\mathcal{A}$  with $\omega\in (0,1)$ has only $n$ non-unity eigenvalues and 
		\[
		\sigma(\CP^{-1}(\omega)\mathcal{A})\subset \mathbb{A}_\omega:=\left\{ 
		z\in\IC:\ \frac{\omega }{(2-\omega)}\le \left|z-\frac{2}{2-\omega} \right| 
		\le  \frac{\omega }{(2-\omega)(1-\omega)}
		\right\}. 
		\]
	\end{theorem} 
	\begin{proof} 
		  {We first claim that $P_\mu(\omega)=I_m+ \mu S(\omega) D$  is nonsingular for  $\mu\in{\rm i}\mathbb{R}$, which implies that the \pdd $\CP(\omega)$ is invertible.  If $\mu=0$, the claim holds trivially. Hence we only have to consider  $\mu\ne 0$.  Since $\mu\in\Ii\mathbb{R}$, it is sufficient to prove that any eigenvalue of  $S(\omega)D$ has non-zero real part.  This is further equivalent to proving that any  eigenvalue of $D^{\frac{1}{2}}S(\omega)D^{\frac{1}{2}}$ has non-zero real part, since $S(\omega)D$ is similar to   $D^{\frac{1}{2}}S(\omega)D^{\frac{1}{2}}$. 	 Since  $S$ is a kew-symmetric matrix (and thus $S$ is diagonalizable) and  $S(\omega)=S+\frac{1-\omega}{2}e_m^\T e_m$  is a rank-one perturbation of $S$, from \cite[Theorem 2.3]{2011Eigenvalue} we know that $S(\omega)$ is  diagonalizable for any $\omega\in(0, 1)$.  This implies that the eigenvectors of $S(\omega)$, denoted by $\{v_1, v_2,\dots, v_m\}$, forms a basis of $\mathbb{C}^{m}$.   Without loss of generality, we assume that these eigenvectors are  orthonormal basis of $\mathbb{C}^m$ (after the Gram-Schmidt orthogonalization and normalization).  Moreover,  we assume that the eigenvalue associated with $v_j$ is $\lambda_j$. From  \cite[Theorem 2.1]{han2014proof} we know  that these eigenvalues   lie in the open right half-plane, that is $\Re(\lambda_j)>0$ for $j=1,2,\dots, m$.   Hence, for any vector $z\in\mathbb{C}^m$ expressed as $z=c_1v_1+c_2v_2+\cdots+c_mv_m$ it holds that
		    \begin{equation}\label{ReS}
		    \Re(z^*S(\omega)z)=\Re(\lambda_1|c_1|^2+\lambda_2|c_2|^2+\cdots+\lambda_m|c_m|^2)>0.  
		    \end{equation}
		 Let $(\lambda, \xi)$ is an eigenpair of $D^{\frac{1}{2}}S(\omega)D^{\frac{1}{2}}$, i.e., 
		  $D^{\frac{1}{2}}S(\omega)D^{\frac{1}{2}}\xi=\lambda \xi$, with $\xi\neq0$. 
		   we have
		   $$
		   \xi^*D^{\frac{1}{2}}S(\omega)D^{\frac{1}{2}}\xi=\lambda\|\xi\|_2^2.
		   $$ Now, by letting $z=D^{\frac{1}{2}}\xi$ in \eqref{ReS} it follows that  $\Re(\xi^*D^{\frac{1}{2}}S(\omega)D^{\frac{1}{2}}\xi)>0$ and therefore $\lambda$ has non-zero real part.}

		Let   $z(\mu)=e_m^\T\left(\mu^{-1}D^{-1}+I^{(-1)}\right)^{-1} e_m$. Using the same notations in Theorem \ref{ThmPAheat_omega},
		 we have
		$$
		\sigma(\CP^{-1}(\omega)\mathcal{A})= {\bigcup}_{-\mu\in\sigma(K)}\sigma( P^{-1}_\mu(\omega) A_\mu)
		=\underbrace{\left\{1,1,\cdots,1\right\}}_{n(m-1)} {\bigcup} \left\{\frac{2}{2-\omega z(\mu)}: -\mu\in\sigma(K)\right\}, 
		$$
{We next prove the following relationship for $\mu\in\Ii\mathbb{R}$
\begin{equation}\label{eq3.13}
 |z(\mu)-1|=1. 
\end{equation}
 Let $w:=\left(\mu^{-1}D^{-1}+I^{(-1)}\right)^{-1} e_m$, which gives
		 	 	 $\left(\mu^{-1}D^{-1}+S+\frac{1}{2}e_m e_m^\T \right) w=e_m$.
		 	 	 Multiplying from left by the conjugate transpose $w^*$, we get
		 	 	 \[
		 	 	  \mu^{-1}w^*D^{-1}w+w^*Sw+\frac{1}{2}w^*e_m e_m^\T w =w^*e_m.
		 	 	 \]
		 	 	Notice that $z(\mu)=e_m^\T w=(w^*e_m)^*$ , we get
		 	 	\[
		 	 	\underbrace{\mu^{-1}w^*D^{-1}w+w^*Sw}_{=:\rho \Ii}+\frac{1}{2}z^*(\mu) z(\mu) =z^*(\mu),
		 	 	\]
		 	 	where  $(\mu^{-1}w^*D^{-1}w+w^*Sw)=\rho \Ii$ with $\rho\in\IR$ holds because   $\mu\in\Ii\IR$ and   $S$ is skew-symmetric. Let $z(\mu)=\Re(z(\mu))+\Ii\Im(z(\mu))$. Then the above equation gives
		 	 	\[
		 	 	\rho\Ii+\frac{1}{2} \Re(z(\mu))^2+\frac{1}{2}\Im(z(\mu))^2=\Re(z(\mu))-\Ii\Im(z(\mu)),
		 	 	\]
		 	 	which leads to (by matching  the real part) $\Re(z(\mu))^2+\Im(z(\mu))^2=2\Re(z(\mu))$, 
		 	 	that is
		 	 	\[
		 	 	(\Re(z(\mu))-1)^2+\Im(z(\mu))^2=1,
		 	 	\]
		 	 	i.e., the relationship \eqref{eq3.13} holds.  }

{By \eqref{eq3.13},  for any $\mu\in\Ii\mathbb{R}$ we have}    $z(\mu)-1=e^{i\theta}$ with some $\theta\in [0,2\pi]$.
		Hence 
		\bean
		\frac{\omega }{(2-\omega)}\le \left|\frac{2}{2-\omega z(\mu)}-\frac{2}{2-\omega} \right|=
		\frac{|2\omega e^{\i\theta}|}{(2-\omega)|((2-\omega)-\omega   e^{i\theta})|}
		\le  \frac{\omega }{(2-\omega)(1-\omega)},
		\eean
		where we have used $|e^{i\theta}|=1$ and the triangle inequality
		\[
		2-2\omega=(2-\omega)-\omega |e^{i\theta}|\le	|((2-\omega)-\omega   e^{i\theta})|\le (2-\omega)+ \omega |e^{i\theta}|=2.
		\]
		Hence, all the $n$ non-unity eigenvalues of $\CP^{-1}(\omega)\mathcal{A}$ can be bounded by an annulus centered at $(\frac{2}{2-\omega},0)$ with outer radius $\frac{\omega }{(2-\omega)(1-\omega)}$ and inner radius $\frac{\omega }{(2-\omega)}$, that is
		\[
		\sigma(\CP^{-1}(\omega)\mathcal{A})\subset \mathbb{A}_\omega:=\left\{ 
		z\in\IC:\ \frac{\omega }{(2-\omega)}\le \left|z-\frac{2}{2-\omega} \right| 
		\le  \frac{\omega }{(2-\omega)(1-\omega)}
		\right\}, 
		\]
		which  together with $1\in \mathbb{A}_\omega$ completes the proof.
	\end{proof}

{In Figures \ref{eigplotheat_omega} and   \ref{eigplotwave_omega} we  
 plot the eigenvalues of $\CP^{-1}(\omega)\CA$ for the linear heat equation  and   wave equation with $\omega=0.1$,  respectively. For the heat equation,   all the eigenvalues are real and located within an  interval $[1,1.11)$  as estimated in Theorem \ref{ThmPAheat_omega}.  For the wave equation, we see from Figure \ref{eigplotwave_omega} that   the eigenvalues are   located within an annulus with a very narrow bandwidth.   
In both cases, the spectrum $\sigma(\CP^{-1}(\omega)\CA)$ is uniformly bounded and clustered around $1$ and the results indicate that  the estimates in Theorems \ref{ThmPAheat_omega} and   \ref{ThmPAwave_omega}  are sharp.}  

\begin{figure}[htp!]
	\centering   
	\includegraphics[height=2.5in,width=5.65in]{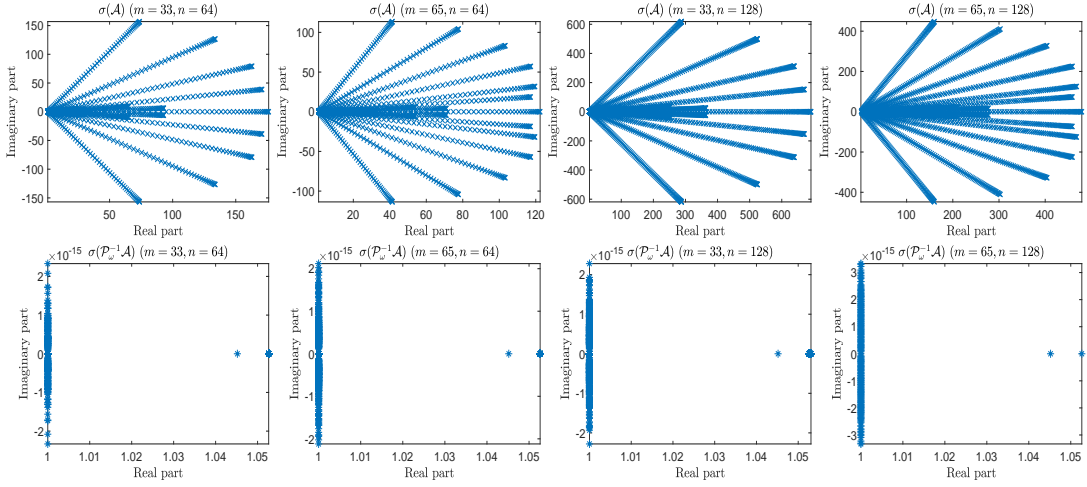} 
	\caption{The eigenvalues  of $\CA$ and $\CP^{-1}(\omega)\CA$ with $\omega=0.1$  {for  1D heat equation with $n=64,128$ and $m=33,65$}.  {Notice {that} all the eigenvalues  of $\CP^{-1}(\omega)\CA$ are real and located in the interval $[1,\frac{1}{1-\omega})=[1,\frac{1}{0.9})$.}} \label{eigplotheat_omega}
\end{figure} 

\begin{figure}[htp!]
	\centering   
	\includegraphics[height=3in,width=5.65in]{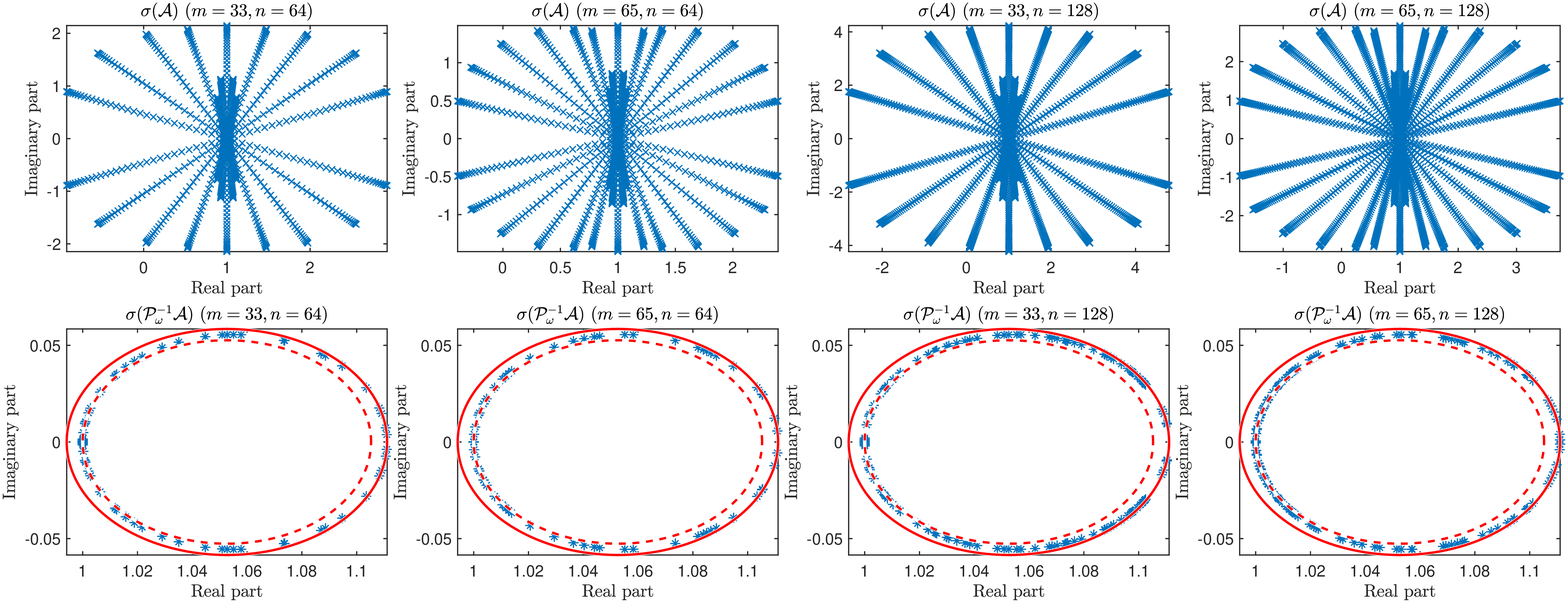} 
	\caption{The eigenvalues  of $\CA$ and $\CP^{-1}(\omega)\CA$ with $\omega=0.1$  {for  1D wave equation with $n=64,128$ and $m=33,65$}.  {Notice {that} all the eigenvalues  of $\CP^{-1}(\omega)\CA$ are locate within an annulus $\mathbb{A}_\omega$ (with center $(\frac{2}{2-\omega},0)$, inner circle radius $\frac{\omega }{(2-\omega)}$ in dashed line, outer circle radius $\frac{\omega }{(2-\omega)(1-\omega)}$ in solid line)}.} \label{eigplotwave_omega}
\end{figure} 

	\subsection{Time-varying and nonlinear case: {NKPA technique}}
	For the linear case with time-varying coefficient matrix $K(t)$, the {all-at-once    matrix  for the Sinc-Nystr\"{o}m method is}
	\begin{align} \label{ODEIVPSinc_A}
		\CA =I_m\otimes I_n-((I^{(-1)}D)\otimes I_n)\BK,
	\end{align}
	{where $\BK={\rm blkdiag}(K(t_{-{M}}),\cdots,K(t_{{M}}))$. 
	To get a diagonalization-based   PinT preconditioner, 
	the widely used approach is to follow the idea in  \cite{gander2017time} to construct an approximation (of  tensor structure)  to $\BK$, such  as}
	\bea \label{NKP}
	\BK\approx I_m\otimes \overline{K}, ~ \overline{K}:=\frac{1}{m}{{\sum}}_{j=-{M}}^{M} K(t_{j}),
	\eea
	{where $m=2{M}+1$}. 
	This leads to \pd 
	\begin{align} \label{ODEIVPSinc_P1}
		\overline\CP =I_m\otimes I_n-(SD\otimes I_n)(I_m\otimes \overline{K} ) 
		=I_m\otimes I_n-(SD\otimes \overline{K} ),
	\end{align}
	{which is of  the same structure as $\CP$ in (\ref{pre_P})} {and therefore the \dg-based PinT procedure \eqref{3step2} is applicable to   $\overline\CP$ as well}.  
	Such an {\textit{averaging}}-based Kronecker product approximation works well when $K(t)$ does not change dramatically over the considered time interval. {However, if $K(t)$ has very large variance on the Sinc time points, using the \pdd $\overline\CP$ may result in slow convergence rate or even divergence for the GMRES method. Here we propose another approximation of $\BK$ based on the  nearest Kronecker product approximation (NKPA) technique. The idea lies in approximating $\BK$ by a tensor structure matrix $\widehat{D}\otimes\overline K$  with diagonal matrix $\widehat{D}$ fixed by  }	
\begin{equation}\label{eq3.17}
	\min_{\widehat D \tn{ is diagonal}}\| \BK- \widehat D\otimes \overline K\|,
\end{equation}
	where  
	$\overline K$ is the averaging  matrix  given in \eqref{NKP}. {Under the Frobenius norm $\|\cdot\|_F$,  a}ccording to \cite[Thm. 3]{van1993approximation} the solution   $\widehat{D}={\rm diag}(D_1,\dots, D_m)$  of   \eqref{eq3.17} has an explicit formula 
	\begin{equation}  \label{Djj}
	\widehat D_j=\frac{{\rm trace}(K(t_j)^\T \overline K)}{{\rm trace}(\overline K^\T \overline K)},\quad j=1, 2, \dots, m, 
	\end{equation}
	where ${\rm trace}(\overline K^\T \overline K)>0$ is assumed.    This gives  the following \pd
	\begin{align} \label{ODEIVPSinc_P2}
		\widehat\CP =I_m\otimes I_n-(SD\otimes I_n) ( \widehat D\otimes \overline K)
		=I_m\otimes I_n-(SD\widehat D \otimes \overline{K} ),
	\end{align}
	which {is also of  the same structure as $\CP$ in (\ref{pre_P})} {since $D\widehat{D}$ is a diagonal matrix}. {Numerically, we find that such an improved preconditioner often  results in}  a significantly faster convergence rate than the averaging-based \pdd $\overline{\CP}$, especially when $\widehat D$  deviates largely from the identity matrix $I_m$ (i.e., $K(t)$ undergoes a large variance over the Sinc collation time points). For the case that $K(t)=K$ is constant,     $\overline{\CP}$ and $\widehat{\CP}$ are {identical}.
	Clearly, the  \dgg procedure (\ref{3step2})   {also applies  to}    $\widehat{\CP}$. However, the spectrum analysis  of {the} preconditioned matrices $\overline{\CP}^{-1}\CA$
	and $\widehat{\CP}^{-1}\CA$    {becomes extremely   difficult and further discussion on this is beyond the scope of this paper.} 
	Analogous to the definition of $\CP(\omega)$, we can also define $\widehat\CP(\omega)$ as  
	\begin{align} \label{ODEIVPSinc_P2_omega}
		\widehat\CP(\omega) =I_m\otimes I_n-(S(\omega)D\otimes I_n) ( \widehat D\otimes \overline K)
		=I_m\otimes I_n-(S(\omega)D\widehat D \otimes \overline{K} ),
	\end{align}
which is expected to perform better than $\widehat\CP$ when $\omega\in (0,1)$ is small.
%	{One may ask whether we can expect   additional improvements or not, by considering  a} more general (and expensive) {NKPA}:
%	\[
%	\min_{\tilde D,\ \tilde K}\| \BK- \tilde D\otimes \tilde K\|_{F},
%	\]
%	{that is to optimize  $\tilde D$ and $\tilde K$   simultaneously}. Interestingly, this minimization problem can be solved from the SVD of a permuted version of the block diagonal matrix $\BK$ \cite{van1993approximation,van2000ubiquitous}.
%	{Nevertheless, we find numerically that  the preconditioner based on} such a more  general  {NKPA} shows almost the same convergence rate  as the {one  $\widehat\CP$ in \eqref{ODEIVPSinc_P2}}  and hence {we stop further exploration  on  this}. 
	
	In the nonlinear case, the block-diagonal matrix $\mathsf{Q}(\bm y_h)$ in (\ref{Qh})   shares the same block-diagonal structure as $\BK$ and thus the aforementioned NKPA technique can be used, too.   This observation naturally  leads to a similar  NKPA-based  \pdd   for  the GMRES  method   used as an  inner solver  within the Newton iteration  \eqref{Newton}  for solving   the Jacobian system.   
	
		\section{Numerical examples}\label{sec5}
	In this section, we  {present numerical results} to illustrate the effectiveness of our proposed {\pd}s.
	All simulations are implemented using MATLAB on {a Dell Precision 5820 Tower Workstation with Intel(R) Core(TM) i9-10900X 3.70GHz CPU and 64GB RAM}.
	The CPU time (in seconds) is estimated  by using the timing functions \texttt{tic}/\texttt{toc},   based on the serial implementation of the preconditioned  iterative algorithms. We employ the right-preconditioned GMRES \cite{saad1986gmres} solver (without restarts) in the IFISS package \cite{ifiss,ers07,ers14}, 
	and choose a zero initial guess and a small stopping tolerance $\texttt{\tol}=10^{-10}$ (for high order accuracy purpose) based on the reduction in relative residual norms.  {The number of GMRES iterations for achieving the stopping tolerance is denoted by It$_{\rm G}$}. We will take $d=\pi/2$ and $\alpha=1$ in the {Sinc-Nystr\"{o}m}  method.
	In measuring the accuracy of the {Sinc-Nystr\"{o}m}  method, we will report the maximum error (denoted by `Error') between Sinc approximation and the exact solution (if known) over all non-uniform Sinc time points. 
		
	For both the heat and wave equations,  We discretize the Laplacian operator $\Delta$ by a {centered}  difference scheme in space with a uniform mesh step  size $h_x$ to get the discrete Laplacian matrix  $\Delta_{h_x}$. {For all numerical experiments, the \pd{s} proposed in this paper are used according to the \dgg procedure \eqref{3step2}.}		In rectangular domains with regular grids,
		the complex-shifted systems in Step-(ii) of \eqref{3step2}   are  solved in serial by MATLAB's sparse direct solver (Thomas algorithm)  and fast Poisson direct  solver \cite{saad2003iterative} (based on discrete sine transform) for 1D and 2D cases, respectively.
	 		For  more general domains with irregular grids (e.g., finite element discretization), fast iterative solvers  {(e.g., the  multigrid method \cite{2011Parallel,2018Solving,2021HG}, the domain decomposition method \cite{2015Domain} and the preconditioned GMRES method \cite{2015Applying}) can be used}.

	\subsection*{Example 1: linear 2D heat equation with constant coefficients}
{{We first consider the following}}
 2D heat equation defined on the space domain  $\Omega=(0,\pi)^2$:
	\begin{equation}\label{heatPDE}
		\begin{cases}
			y_{t}= \Delta y +g, &\tn{in} \Omega\times(0,T),\\
			y=0,&\tn{on} \partial\Omega\times(0, T), 
		\end{cases}
	\end{equation} 
	{where  the initial condition $y(x_1,x_2,0)=x_1(\pi-x_1)x_2(\pi-x_2)$ and source term $g$ are chosen such that the  exact solution is} $y(x_1,x_2,t)=x_1(\pi-x_1)x_2(\pi-x_2)e^{-t}$. 
	Table \ref{T2heat_2D} shows the error and convergence results   for the GMRES  method  without \pdd (denoted as `None') and with our PinT {\pd}s $\CP$ and $\CP(\omega)$  {(with $\omega=0.01$)}, respectively. 
	With the \pdd $\CP$ only a few iterations is sufficient to achieve {stopping tolerance}. 
	 Such a fast convergence rate is anticipated from the highly clustered spectrum distribution of $\CP^{-1}\CA$ given in Theorem \ref{ThmPA}.  
	Interestingly, for a fixed $n$ (e.g. $n=32^2$) we do observe that ${\rm It_G}$ slightly decreasing as $m$   increases, which is reasonable since $\CP^{-1}\CA$ has only $n$  non-unity eigenvalues regardless of $m$.
	For this example, the improved \pdd $\CP(\omega=0.01)$ shows almost the same convergence rate as $\CP$,
	which is anticipated since the spectrum of $\CP^{-1}\CA$ is already highly clustered,
	 as shown in Figures \ref{fig1eigplot} and \ref{eigplotheat_omega}.
	 Nevertheless, for $m=257$ the Error corresponding to $\CP(\omega=0.01)$ seems to slightly larger than that by $\CP$,
	 which is due to a larger roundoff error during \dg.

	\begin{table}[htp!]\small
		\centering 
		\caption{Results for Example 1 (2D heat PDE with constant coefficients, $T=2$,$\texttt{\tol}=10^{-10}$)}
		\begin{tabular}{|c|c||ccc||ccc|ccc|ccc|c|}\hline
			&&\multicolumn{3}{c|}{None}  &\multicolumn{3}{c|}{  $\CP$}&\multicolumn{3}{c|}{  $\CP(\omega)$ with $\omega=0.01$}
			\\
			\hline
			$n$	&$m$& Error & ${\rm It_G}$ &CPU&Error & ${\rm It_G}$ &CPU&Error & ${\rm It_G}$ &CPU\\   \hline 
			
			\multirow{4}{*}{$32^2$}	 
			&33 	 &1.3e-03&	  682&	15.30    &1.3e-03&	 	 4&	0.04&1.3e-03&	 	 3&	0.03\\
			&65 	 &3.5e-05&	  663&	23.52    &3.5e-05&	 	 3&	0.06&3.5e-05&	 	 3&	0.06\\
			&129 	 &2.0e-07&	  596&	177.73   &2.1e-07&	 	 3&	0.15&2.0e-07&	 	 3&	0.16\\
			&257 	 &5.0e-09&	  531&	230.47   &2.9e-10&	 	 3&	0.34&4.2e-08&	 	 3&	0.33\\
			\hline
			\multirow{4}{*}{$64^2$}	 
			&33 &&$>$1000&   &1.3e-03&		 4&	0.14&1.3e-03&	 	 3&	0.11\\
			&65 &&$>$1000&   &3.5e-05&		 3&	0.21&3.5e-05&	 	 3&	0.21\\
			&129&&$>$1000&   &2.1e-07&		 3&	0.45&2.1e-07&	 	 3&	0.43\\
			&257&&$>$1000&   &2.9e-10&		 3&	1.03&4.4e-08&	 	 3&	1.02\\
			\hline
			\multirow{4}{*}{$128^2$}	 
			&33 &&$>$1000&   &1.3e-03&		 5&	0.61&1.3e-03&	 	 3&	0.38\\
			&65 &&$>$1000&   &3.5e-05&		 3&	0.81&3.5e-05&	 	 3&	0.78\\
			&129&&$>$1000&   &2.1e-07&		 3&	1.73&2.1e-07&	 	 3&	1.74\\
			&257&&$>$1000&   &2.9e-10&		 3&	4.02&5.2e-08&	 	 3&	4.06\\
			
			\hline
			%			\multirow{4}{*}{$256^2$}	 
			%		&33 &&$>$500&  &1.30e-03&		 6&	3.862\\
			%		&65 &&$>$500&  &3.54e-05&		 3&	4.609\\
			%		&129&&$>$500&  &2.05e-07&		 3&	9.877\\
			%		&257&&$>$500&  &2.80e-10&		 3&	24.276\\
			%			\hline
		\end{tabular}
		\label{T2heat_2D}
	\end{table}
	
		\subsection*{Example 2: linear 2D heat equation with time-varying coefficients}
	{We next c}onsider a linear 2D heat equation with time-varying coefficient on $\Omega=(0,\pi)^2$:
	\begin{equation}\label{heatPDEtime}
		\begin{cases}
			y_{t}=\kappa(t) \Delta y, &\tn{in} \Omega\times(0,T),\\
			y=0,&\tn{on} \partial\Omega\times(0, T),\\ 
		\end{cases}
	\end{equation} 
	where $\kappa(t)=1/\left( (1.2+t)\ln(1.2+t)\right)$ and the initial condition $y(\cdot,0)$ is chosen such that the exact solution {is} $y(x_1,x_2,t)=x_1(\pi-x_1)x_2(\pi-x_2)/{\ln(1.2+t)}$. 
		{In} Figure \ref{Ex5eigplot} {we plot  the eigenvalues} of $\CA$, $\overline\CP^{-1}\CA$ and $\widehat\CP^{-1}\CA$ for a fixed space-time mesh  ($m=33$ and $n=16^2=256$). {For $\overline\CP$ we compute the    diagonal matrix  $\widehat D$ according to  the formula  \eqref{Djj} and from  Figure \ref{Ex5eigplot} on the top right we see that such a diagonal matrix} is indeed  very different from an identity matrix. {From the two subfigures on the bottom row we see that} the eigenvalues of $\widehat\CP^{-1}\CA$ are more clustered than that of $\overline\CP^{-1}\CA$.  
		\begin{figure}[htp!]
		\centering
		\includegraphics[width=0.9\textwidth]{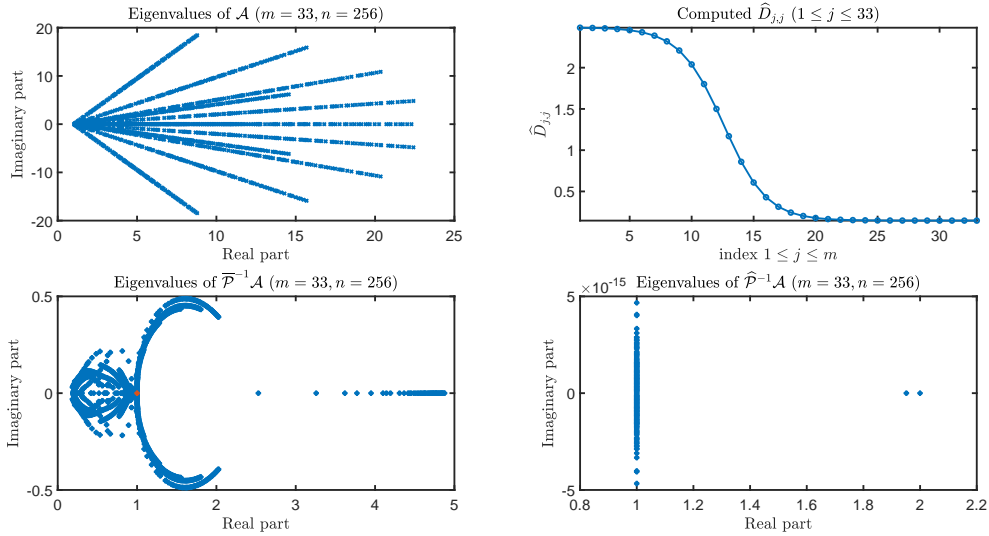} 
		\caption{The eigenvalue distribution of $\CA$, $\overline\CP^{-1}\CA$, and $\widehat\CP^{-1}\CA$ and   $\widehat D$ for Example 2 (2D time-varying case).} \label{Ex5eigplot}
	\end{figure}

	{In}
	Table \ref{TabHeat2Dtime}  {we} report   the errors and convergence results {for the GMRES method without \pdd and with two    PinT preconditioners $\overline\CP$ and $\widehat{\CP}$. It is clear that   the NKPA-based \pdd $\widehat{\CP}$ {in  \eqref{ODEIVPSinc_P2}} leads to} faster convergence rate than the averaging-based \pdd $\overline{\CP}$ given by  \eqref{ODEIVPSinc_P1}. {This result confirms very well the eigenvalue distribution of $\overline\CP^{-1}\CA$ and $\widehat\CP^{-1}\CA$ in Figure \ref{Ex5eigplot}. We also tested the GMRES method using the \pdd $\widehat{\CP}(\omega)$ with $\omega=0.01$, but the results are very similar to that of $\widehat{\CP}$. So we omit the presentation.} 
	\begin{table}[htp!] \small
		\centering 
		\caption{Results for  Example 2 (2D heat PDE with time-varying coefficients, $T=2, \tol=10^{-10}$)}
		\begin{tabular}{|c|c||ccc|ccc|ccc|ccc|ccc|c|}\hline
			&	&\multicolumn{3}{c|}{None} &\multicolumn{3}{c|}{$\overline{\CP}$} &\multicolumn{3}{c|}{$\widehat{\CP}$} 			  
			\\
			\hline
			$n$&	$m$& Error   & ${\rm It_G}$ &CPU&Error   & ${\rm It_G}$ &CPU& Error  & ${\rm It_G}$ &CPU  \\   \hline 
			\multirow{4}{*}{$32^2$} 
&33 	 &3.2e-02&	 	 669&	15.48    &3.2e-02&	 	 69&	0.76   &3.2e-02&	 	 4&	0.05    \\
&65 	 &8.9e-04&	 	 628&	21.92    &8.9e-04&	 	 68&	1.52   &8.9e-04&	 	 3&	0.08    \\
&129	 &5.1e-06&	 	 552&	148.42   &5.1e-06&	 	 69&	5.06   &5.1e-06&	 	 3&	0.17    \\
&257	 &2.7e-08&	 	 488&	223.90   &1.5e-08&	 	 69&	10.66  &6.7e-09&	 	 3&	0.41    \\

\hline
\multirow{4}{*}{$64^2$}		  
&33 &&$>$1000&   &3.2e-02&	 	 73&	4.59    &3.2e-02&	 	 4&	0.18  \\
&65 &&$>$1000&   &8.9e-04&	 	 71&	8.32    &8.9e-04&	 	 3&	0.29  \\
&129&&$>$1000&   &5.1e-06&	 	 71&	15.05   &5.1e-06&	 	 3&	0.60  \\
&257&&$>$1000&   &1.7e-08&	 	 71&	34.35   &6.7e-09&	 	 3&	1.38  \\
\hline
\multirow{4}{*}{$128^2$}	  
&33 &&$>$1000&    &3.2e-02&	  81&	18.45   &3.2e-02&	 	 5&	0.78   \\
&65 &&$>$1000&    &8.9e-04&	  72&	30.52   &8.9e-04&	 	 4&	1.35   \\
&129&&$>$1000&    &5.1e-06&	  72&	64.18   &5.2e-06&	 	 3&	2.46   \\
&257&&$>$1000&    &1.7e-08&	  72&	140.85  &6.6e-09&	 	 3&	5.50   \\
			\hline 
		\end{tabular}
		\label{TabHeat2Dtime}
	\end{table}

	\subsection*{Example 3: linear 2D wave equation} 
{We now c}onsider a linear 2D  wave equation defined on  $\Omega=(0,\pi)^2$:
	\begin{equation}\label{wavePDE}
		\begin{cases}
			y_{tt}=\Delta y+g, &\tn{in} \Omega\times(0,T),\\
			y=0,&\tn{on} \partial\Omega\times(0, T),\\
			%y(\cdot, 0)=\sin(x)\cos(t),&\tn{in} \Omega,\\
			%y_t(\cdot, 0)=-\sin(x)\sin(t),&\tn{in} \Omega.
		\end{cases}
	\end{equation} 
	where the initial conditions $y(\cdot, 0)$ and $y_t(\cdot, 0)$ are fixed  according to   the   exact solution   $y(x_1,x_2,t)=x_1(\pi-x_1)x_2(\pi-x_2)\ln(1+t).$
	By defining $p=y_t$, {this} second-order wave equation can be {reduced to}  a first-order PDE system:
	\begin{equation}\label{wavePDE2}
		\begin{cases}
			y_t=p; &\tn{in} \Omega\times(0,T),\\
			p_{t}=\Delta y+g, &\tn{in} \Omega\times(0,T),\\
			y=0, p=0&\tn{on} \partial\Omega\times(0, T).\\
			%y(\cdot, 0)=\sin(x)\cos(t),&\tn{in} \Omega\\
			%p(\cdot, 0)=-\sin(x)\sin(t),&\tn{in} \Omega.
		\end{cases}
	\end{equation} 
	{By applying the centered   finite difference   in space with a uniform mesh step size $h_x$ to \eqref{wavePDE2}}, we   obtain a linear ODE system with a $2n\times 2n$ sparse constant coefficient matrix
	$$
	K=\bmt \bm 0 & I_n\\ \Delta_{h_x} & \bm 0\emt \in \IR^{2n\times 2n}.
	$$
	In
		Table \ref{T2wave}, we present  the   errors and convergence results of the GMRES method without  \pdd and with the proposed \pdd $\CP$ and $\CP(\omega)$ (with $\omega=0.01$). 
	In contrast to the above heat equations, {the   GMRES method without \pdd does not converge within 1000 iterations for all combinations of $n$ and $m$}.   
Fortunately, the improved preconditioner $\CP(\omega)$ with {a moderate parameter}  $\omega=0.01$ can achieve much faster mesh-independent convergence rates, which confirms   Theorem \ref{ThmPAwave_omega} very well.
	\begin{table}[htp!]\small
		\centering 
		\caption{Results for Example 3 (2D wave PDE, $T=2$,$\texttt{\tol}=10^{-10}$)}
		\begin{tabular}{|c|c||ccc||ccc|ccc|c|}\hline
			&&\multicolumn{3}{c|}{None}   &\multicolumn{3}{c|}{$\CP(\omega)$ with $\omega=0.01$}
			\\
			\hline
			$2n$	&$m$& Error & ${\rm It_G}$ &CPU &Error & ${\rm It_G}$ &CPU\\   \hline 
  
			\multirow{4}{*}{$2\times 32^2$}	 
&33 	&&	 	 $>$1000&                    &1.8e-03&	 	 5&	0.07\\
&65 	&&	 	 $>$1000&                  &4.9e-05&	 	 5&	0.16\\
&129     &&$>$1000&                        &2.8e-07&	 	 5&	0.39\\
&257     &&$>$1000&                      &4.1e-10&	 	 5&	0.78\\
\hline
\multirow{4}{*}{$2\times 64^2$}	
&33 	&&	 	 $>$1000&          &1.8e-03&	 	 5&	0.25\\
&65 	&&	 	 $>$1000&       &4.9e-05&	 	 5&	0.55\\
&129     &&$>$1000&            &2.8e-07&	 	 5&	1.20\\
&257     &&$>$1000&           &4.3e-10&	 	 5&	2.76\\ 
			\hline	 
		\end{tabular}
		\label{T2wave}
	\end{table}
	
%	\begin{table}[htp!]\small
%		\centering 
%		\caption{Results for Example 3 (2D wave PDE, $T=2$,$\texttt{\tol}=10^{-10}$)}
%		\begin{tabular}{|c|c||ccc||ccc|ccc|ccc|c|}\hline
%			&&\multicolumn{3}{c|}{None}  &\multicolumn{3}{c|}{$\CP$}&\multicolumn{3}{c|}{$\CP(\omega=0.01)$}
%			\\
%			\hline
%			$2n$	&$m$& Error & ${\rm It_G}$ &CPU&Error & ${\rm It_G}$ &CPU&Error & ${\rm It_G}$ &CPU\\   \hline 
%  
%			\multirow{4}{*}{$2\times 32^2$}	 
%&33 	&&	 	 $>$1000&                &1.8e-03&		 229&	5.29    &1.8e-03&	 	 5&	0.07\\
%&65 	&&	 	 $>$1000&                &4.9e-05&		 217&	28.98   &4.9e-05&	 	 5&	0.16\\
%&129     &&$>$1000&                      &2.8e-07&		 239&	61.76   &2.8e-07&	 	 5&	0.39\\
%&257     &&$>$1000&                      &2.7e-09&		 213&	87.41   &4.1e-10&	 	 5&	0.78\\
%\hline
%\multirow{4}{*}{$2\times 64^2$}	
%&33 	&&	 	 $>$1000&       &1.8e-03&	  503&	231.41    &1.8e-03&	 	 5&	0.25\\
%&65 	&&	 	 $>$1000&       &4.9e-05&	  462&	320.96    &4.9e-05&	 	 5&	0.55\\
%&129     &&$>$1000&             &2.8e-07&	  487&	659.69    &2.8e-07&	 	 5&	1.20\\
%&257     &&$>$1000&             &2.4e-09&	  453&	1193.21   &4.3e-10&	 	 5&	2.76\\ 
%			\hline	 
%		\end{tabular}
%		\label{T2wave}
%	\end{table}
	
	{The \pdd $\CP(\omega)$ contains a free parameter $\omega$ and in Figure \ref{CondV_omega} we have plotted the condition number of the eigenvector matrix of $S(\omega)D$ for different values of $\omega$. As we mentioned there, such a condition number is proportional to   the roundoff error arising from the \dgg procedure \eqref{3step2} and a large roundoff error will seriously pollute the \diss accuracy.  So, it would be interesting to illustrate how the parameter $\omega$ affects  \diss accuracy  in practice. In}  Table \ref{T3wave} we report  the  errors and convergence results of the GMRES method with  NKPA-based \pdd  $\CP(\omega)$ for a set of  values of $\omega$. 
	 We see that {for the first few $\omega$  the iteration number  decreases  as $\omega$ decreases, 
	 but it  re-bounces  when $\omega\le 10^{-9}$.   From the results for Error, we see that   the roundoff error  due to the \dgg procedure   
	 quickly contaminate the \diss accuracy.  In particular, for $\omega=10^{-15}$ the measured Error is very bad. This can be explained as follows. For such a small $\omega$, $\CP(\omega)$  approximately equals to      $I^{-1}D$, so applying the \dgg procedure to $\CP(\omega)$ is   equivalent to diagonalizing $I^{-1}D$ (as $I^{-1}D=U\Psi U^{-1}$), which is unstable due to a very large condition number   of $U$ (cf. Figure \ref{CondV_omega}).  For this example, 
	 it seems  $\omega\approx 10^{-6}$ is the best choice}.  
		\begin{table}[htp!] \small
		\centering 
		\caption{Results for Example 3 with the improved preconditioner $\CP(\omega)$ (2D wave PDE, $T=2$,$\texttt{\tol}=10^{-10}$)}
		\begin{tabular}{|c|c||cc||cc||cc||cc||cc|cccc|}\hline
			&&\multicolumn{2}{c|}{$\omega=10^{-3}$}  & \multicolumn{2}{c|}{$\omega=10^{-6}$}  & \multicolumn{2}{c|}{$\omega=10^{-9}$}  & \multicolumn{2}{c|}{$\omega=10^{-12}$} & \multicolumn{2}{c|}{$\omega=10^{-15}$}
			\\
			\hline
			$2n$	&$m$& Error & ${\rm It_G}$& Error & ${\rm It_G}$& Error & ${\rm It_G}$& Error & ${\rm It_G}$ &{Error} & ${\rm It_G}$ \\   \hline 
			 
	\multirow{4}{*}{$2\times 32^2$}	  
&33 	&1.8e-03&	  3&	 1.8e-03&	  2&   1.8e-03&	  2&     1.8e-03&	  2&	1.8e-03&	  2      \\
&65 	&4.9e-05&	  3&	 4.9e-05&	  2&   4.9e-05&	  2&     1.3e-04&	  3&	7.1e-05&	  3      \\
&129    &2.8e-07&	  3&	 2.8e-07&	  2&   9.7e-07&	  2&     3.6e-04&	  3&	2.3e-01&	  10     \\
&257    &9.9e-10&	  3&	 1.6e-09&	  2&   1.1e-06&	  3&     7.4e-04&	  35&	7.1e-01&	  10     \\
\hline                                   
\multirow{4}{*}{$2\times 64^2$}	         
&33 	&1.8e-03&	  3&	 1.8e-03&	  2&    1.8e-03&	  2&   1.8e-03&	  2&	  1.8e-03&	  2 	 \\
&65 	&4.9e-05&	  3&	 4.9e-05&	  2&    4.9e-05&	  2&   1.2e-04&	  3&	  7.4e-05&	  3 	 \\
&129    &2.9e-07&	  3&	 2.8e-07&	  2&    9.4e-07&	  2&   4.0e-04&	  3&	  1.8e-01&	  10     \\
&257    &1.0e-09&	  3&	 1.7e-09&	  2&    1.0e-06&	  3&   6.0e-04&	  40&     5.7e-01&	  10     \\
\hline	 
\multirow{4}{*}{$2\times 128^2$}	 
&33 	&1.8e-03&	  3&	 1.8e-03&	  2&   1.8e-03&	  2&   1.8e-03&	  2&	     1.8e-03&	  2 	 \\
&65 	&4.9e-05&	  3&	 4.9e-05&	  2&   4.9e-05&	  2&   1.7e-04&	  3&	     7.6e-05&	  3 	 \\
&129    &2.9e-07&	  3&	 2.8e-07&	  2&   1.1e-06&	  2&   3.8e-04&	  3&	     1.6e-01&	  10 	 \\
&257    &1.1e-09&	  3&	 1.1e-09&	  2&   8.4e-07&	  3&   5.1e-04&	  39&        5.6e-01&	  9 	 \\

			\hline	
		\end{tabular}
		\label{T3wave}
	\end{table}
	
	\subsection*{Example 4:   Allen--Cahn equation} 
	{At the end of this section, we}
	consider the  1D Allen--Cahn    equation \cite{kassam2005fourth} on a spatial domain  $\Omega=(-1,1)$:
	\begin{equation}\label{AllenCahn}
		\begin{cases}
			y_{t}=0.01 y_{xx}+y-y^3, &\tn{in} \Omega\times(0,T),\\
			y(-1,t)=-1,\ y(1,t)=1,&\tn{in} (0,T), \\
			y(x, 0)=0.53x+0.47 \sin(-1.5\pi x),&\tn{in} \Omega.
		\end{cases}
	\end{equation}
	{We first apply the  centered   finite difference scheme with a uniform mesh step size $h_x$ to get a nonlinear ODE system, for which} the  nonlinear Sinc-Nystr\"{o}m  system is solved by  Newton's method    (\ref{Newton}) with zero initial guess,
	where the Jacobian system  for each Newton iteration is   solved by GMRES without \pdd  and with the NKPA-based   \pdd $\widehat\CP$, respectively. 
	In Table \ref{T1AllenCahn}, we  show the errors and iteration numbers   for Newton's method {(denoted by  ${\rm It_{\rm N}}$)} and  the maximal iteration number of the GMRES method  over all the Newton iterations (denoted by It$_{\rm G}$).   	While costing the same number of outer Newton iterations,
	the \pdd  $\widehat\CP$ leads to  much  faster convergence for the GMRES method   and much less CPU time.
	Notice  that ${\rm It_G}$ for GMRES without \pdd  increases  dramatically as the spatial   size $n$ {grows}.  {(The results for the \pdd $\overline{\CP}$ defined by  \eqref{ODEIVPSinc_P1}
 is omitted since it gives the same  ${\rm It_G}$  as $\widehat{P}$,  perhaps due to  small variance  of the solution in   time.)  With a  generalized version of $\widehat\CP$, i.e., $\widehat\CP(\omega)$  with $\omega=0.01$,  we see in Table \ref{T1AllenCahn} that both  ${\rm It_G}$ and the CPU time can be further  reduced. }
 	\begin{table}[htp!] \small
		\centering 
		\caption{Results for Example 1 (1D Allen--Cahn PDE, $T=2$, $\texttt{\tol}=10^{-10})$}
		\begin{tabular}{|c|c||cccc||cccc||cccc|cccc|}\hline
			&&\multicolumn{4}{c|}{None}  &\multicolumn{4}{c|}{$\widehat\CP$} &\multicolumn{4}{c|}{$\widehat\CP(\omega)$ with $\omega=0.01$}
			\\
			\hline
			$n$	& $m$&Error   &${\rm It_{\rm N}}$&  ${\rm It_G}$ &CPU&Error   & ${\rm It_{\rm N}}$& ${\rm It_G}$ &CPU&Error   & ${\rm It_{\rm N}}$& ${\rm It_G}$ &CPU \\   \hline 
					\multirow{4}{*}{256}	
		&33 	  &2.6e-05&	5&	 473&	16.64    &2.6e-05&	5&	 14&	0.22  &2.6e-05&	5&	 7&	0.13\\
		&65 	  &4.7e-07&	5&	 419&	20.28    &4.7e-07&	5&	 14&	0.42  &4.7e-07&	5&	 7&	0.26\\
		&129 	  &1.9e-09&	5&	 362&	24.90    &1.9e-09&	5&	 14&	1.01  &1.9e-09&	5&	 7&	0.55\\
		&257 	  &1.5e-11&	5&	 314&	35.86    &1.4e-11&	5&	 14&	2.36  &1.4e-11&	5&	 7&	1.30\\
		\hline 
		\multirow{4}{*}{512}	
		&33 	&&&$>$1000&   &2.6e-05&	5&	 14&	0.68  &2.6e-05&	5&	 7&	0.50\\
		&65 	&&&$>$1000&   &4.7e-07&	5&	 14&	1.29  &4.7e-07&	5&	 7&	0.97\\
		&129 	&&&$>$1000&   &1.9e-09&	5&	 14&	2.95  &1.9e-09&	5&	 7&	2.06\\
		&257 	&&&$>$1000&   &1.4e-11&	5&	 14&	7.05  &1.4e-11&	5&	 7&	4.68\\
		\hline 
		\multirow{4}{*}{1024}	
		&33 	&&&$>$1000&   &2.6e-05&	5&	 14&	1.99  &2.6e-05&	5&	 7&	1.28\\
		&65 	&&&$>$1000&   &4.7e-07&	5&	 14&	3.95  &4.7e-07&	5&	 7&	3.38\\
		&129 	&&&$>$1000&   &1.9e-09&	5&	 14&	9.28  &1.9e-09&	5&	 7&	7.12\\
		&257 	&&&$>$1000&   &1.4e-11&	5&	 14&	19.50 &1.4e-11&	5&	 7&	14.72\\
			\hline 
		\end{tabular}
		\label{T1AllenCahn}
	\end{table}

	Since the exact solution is unknown, we compute the reference solution by using MATLAB's ODE solver \texttt{ode15s} with a very small tolerance $10^{-12}$ and the same space-time mesh. As expected, the reported errors in Table \ref{T1AllenCahn}  shows an exponential order of accuracy in time.  Figure \ref{ACsolplot} illustrates the reference and approximate solutions,
	where we see clearly how the   non-uniform Sinc mesh points in time cluster near $t=0$ and $t=T$.
	\begin{figure}[htp!]
		\centering
		\includegraphics[width=1\textwidth]{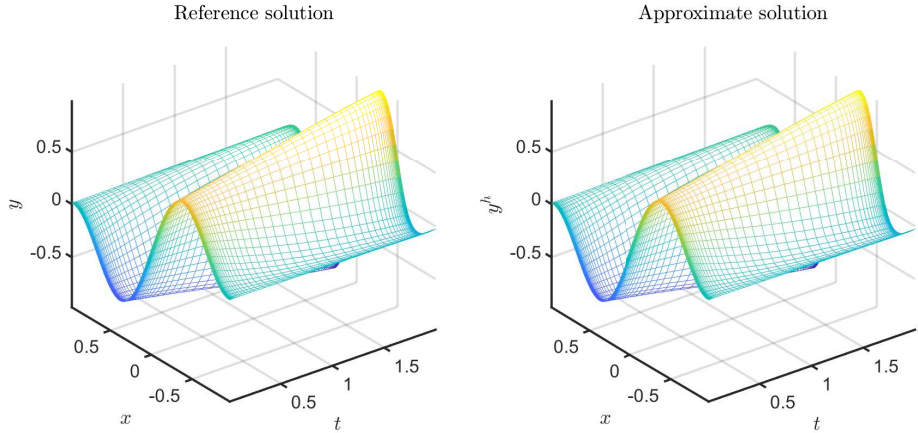} 
		\caption{The reference and Sinc approximation for Example 4 (1D Allen--Cahn PDE) with $m=257, n=64$.} \label{ACsolplot}
	\end{figure}

	\section{Conclusion}
The
	Sinc-Nystr\"{o}m  method for {the initial-value   ODEs}  can achieve   exponential order of accuracy in time {and for this method  the linear  (or nonlinear)  all-at-once system is the major problem  that we need to handle in practice. In this paper, we proposed   some efficient  {\pd}s for solving such an  all-at-once system for both the parabolic and hyperbolic problems.  The construction of the \pdd is based on      looking  insight into  a special structure of the \diss matrix of the Sinc-Nystr\"{o}m  method, namely the   {\em Toeplitz-times-diagonal}  structure.  The spectrum  analysis and the extensive  numerical results  indicate that the preconditioned GMRES method has    mesh-independent convergence rates. Moreover, if parallel computer is available, the proposed \pd{s} can be used in parallel for all the Sinc time points, following a block  \dgg procedure (cf. \eqref{3step2}). We have shown that this idea works, because such a \dgg is well conditioned, i.e., the condition number of the eigenvector matrix of the block \dgg is a moderate quantity and only weakly grows as the number of Sinc time points  increases  (cf. Figure \ref{CondV_omega}).}

{It would be interesting   to generalize this work to  other spectral methods (e.g. Chebyshev method). In the previous work  \cite{2021Spectral,2008Accuracy}, it was shown that these methods can be  very useful   in improving the accuracy  in time of the numerical solutions, but the large scale all-at-once system could be a serious problem  for applying these methods to time-dependent PDEs.  Such a generalization is by no means  trivial,    because the structure of    {all-at-once} matrix is completely different   from  that of the Sinc-Nystr\"{o}m method and therefore the construction of the \pdd and the spectral analysis of the preconditioned matrix need new ideas}.  
\section*{Acknowledgement}    
The authors would like to thank Dr. Xiang-Sheng Wang from University of Louisiana at Lafayette for pointing out a flaw in the proof of Lemma 3.2.

 \bibliographystyle{siam} %siam
\bibliography{Sinc,waveControl} 
%\bibliography{Sinc,waveControl}

\end{document}